\documentclass[final,onefignum,onetabnum]{siamart}

\usepackage{braket,amsfonts}
\usepackage{array}
\usepackage{multirow}
\usepackage{amsfonts,amssymb} 
\usepackage{epsfig,mathrsfs} 

\usepackage[caption=false]{subfig}
\usepackage{pgfplots}

\newsiamthm{claim}{Claim}
\newsiamremark{remark}{Remark}
\newsiamremark{hypothesis}{Hypothesis}
\crefname{hypothesis}{Hypothesis}{Hypotheses}

\usepackage{algorithmic}

\usepackage{graphicx,epstopdf}

\Crefname{ALC@unique}{Line}{Lines}

\usepackage{amsopn}

\usepackage{xspace}
\usepackage{bold-extra}
\usepackage[most]{tcolorbox}

\colorlet{texcscolor}{blue!50!black}
\colorlet{texemcolor}{red!70!black}
\colorlet{texpreamble}{red!70!black}
\colorlet{codebackground}{black!25!white!25}

\lstdefinestyle{siamlatex}{%
  style=tcblatex,
  texcsstyle=*\color{texcscolor},
  texcsstyle=[2]\color{texemcolor},
  keywordstyle=[2]\color{texemcolor},
  moretexcs={cref,Cref,maketitle,mathcal,text,headers,email,url},
}

\tcbset{%
  colframe=black!75!white!75,
  coltitle=white,
  colback=codebackground, 
  colbacklower=white, 
  fonttitle=\bfseries,
  arc=0pt,outer arc=0pt,
  top=1pt,bottom=1pt,left=1mm,right=1mm,middle=1mm,boxsep=1mm,
  leftrule=0.3mm,rightrule=0.3mm,toprule=0.3mm,bottomrule=0.3mm,
  listing options={style=siamlatex}
}

\DeclareTotalTCBox{\code}{ v O{} }
{ fontupper=\ttfamily\color{black},
  nobeforeafter,
  tcbox raise base,
  colback=codebackground,colframe=white,
  top=0pt,bottom=0pt,left=0mm,right=0mm,
  leftrule=0pt,rightrule=0pt,toprule=0mm,bottomrule=0mm,
  boxsep=0.5mm,
  #2}{#1}

\patchcmd\newpage{\vfil}{}{}{}
\title{High-order BDF convolution quadrature for stochastic fractional evolution equations driven by integrated additive noise
\thanks{
\funding{This research was supported by the Science Fund for Distinguished Young Scholars of Gansu Province under Grant No. 23JRRA1020.}}}

\author{Minghua Chen\thanks{School of Mathematics and Statistics, Gansu Key Laboratory of Applied Mathematics and Complex Systems, Lanzhou University, Lanzhou 730000, People$'$s Republic of China. (\email{chenmh@lzu.edu.cn}, \email{shijk20@lzu.edu.cn}, \email{220220934071@lzu.edu.cn}).}
\and Jiankang Shi\footnotemark[2]
\and Zhen Song\footnotemark[2]
\and Yubin Yan\thanks{Department of Physical, Mathematical and Engineering Sciences, University of Chester, Chester CH1 4BJ, UK. (\email{y.yan@chester.ac.uk}).}
\and Zhi Zhou\thanks{Department of Applied Mathematics, The Hong Kong Polytechnic University, Kowloon, Hong Kong. (\email{zhizhou@polyu.edu.hk}).}
}

\headers{High-order BDF for stochastic fractional PDEs}{M. Chen, J. Shi, Z. Song, Y. Yan, and Z. Zhou}

\ifpdf
\hypersetup{ pdftitle={High-order BDF for stochastic fractional PDEs}}
\fi

\begin{document}
\maketitle


\begin{abstract}
The numerical analysis of stochastic time fractional evolution equations presents considerable challenges due to the limited regularity of the model caused by the nonlocal operator and the presence of noise.
 The existing time-stepping methods exhibit a significantly low order convergence rate.
In this work, we introduce a smoothing technique and develop the novel high-order schemes for solving the linear  stochastic fractional evolution equations driven by integrated additive noise.
Our approach involves regularizing the additive noise through an $m$-fold integral-differential calculus, and discretizing the equation using the $k$-step BDF convolution quadrature.
This novel method, which we refer to as the ID$m$-BDF$k$ method, is able to achieve higher-order convergence in solving the stochastic models.
Our theoretical analysis reveals that the convergence rate of the ID$2$-BDF2 method is $O(\tau^{\alpha + \gamma -1/2})$ for $1< \alpha + \gamma \leq 5/2$, and $O(\tau^{2})$ for $5/2< \alpha + \gamma <3$, where $\alpha \in (1, 2)$ and $\gamma \in (0, 1)$ denote the time fractional order and the order of the integrated noise, respectively.
Furthermore, this convergence rate could be improved to $O(\tau^{\alpha + \gamma -1/2})$ for any $\alpha \in (1, 2)$ and $\gamma \in (0, 1)$, if we employ the ID$3$-BDF3 method.
The argument could be easily extended to the subdiffusion model with $\alpha \in (0, 1)$.
Numerical examples are provided to support and complement the theoretical findings.

\end{abstract}

\begin{keywords}
stochastic fractional evolution equation, integrated additive noise, high-order algorithm, error estimate
\end{keywords}


\section{Introduction}\label{Sec1}
In this paper, we study high-order time discretization schemes for solving  the following linear stochastic fractional evolution equation
 driven by integrated additive noise  \cite{GunzburgerSharp2019,JinNumerical2019,KangGalerkin2022},  for $\alpha \in (1, 2)$ and $ \gamma \in (0, 1)$,
\begin{equation}\label{SDE1.1}
 \partial^{\alpha}_{t} \left(u(t) - \upsilon - t b \right) - A u  = \partial^{-\gamma}_{t} \frac{d W(t)}{dt}, \quad t \in (0, T]
 \end{equation}
with the initial values $u(0) = v \in H$, $u_{t} (0) = b \in H$ where $H = L_{2}(\mathcal{O})$ and $ \mathcal{O} \subset \mathbb{R}^{d}$, $d=1, 2, 3$ is a bounded domain with the smooth boundary.
Here $u_{t}$ means the time derivative and the operator $A = \Delta$, $\mathcal{D} (A) = H_{0}^{1} (\mathcal{O}) \cap H^{2}(\mathcal{O})$ and $\Delta$ is the Laplacain operator.
The operators $\partial_{t}^{\alpha}$ and $\partial_{t}^{-\gamma}$ denote the Riemann-Liouville fractional derivative of order $\alpha \in (1,2)$ and integral of order $\gamma \in (0,1)$, respectively, defined by \cite[p.\,62]{PodlubnyFractional1999}
\begin{equation*}\label{RL1.2}
\partial^{\alpha}_{t} u(t) \!=\! \frac{1}{\Gamma(2-\alpha)}\frac{d^{2}}{dt^{2}}\! \int^{t}_{0} {(t-s)}^{1-\alpha}u(s) ds~~{\rm and}~~
\partial^{-\gamma}_{t} u(t)\!=\! \frac{1}{\Gamma(\gamma)}\! \int^{t}_{0} { {(t-s)}^{\gamma-1}u(s)} ds.
\end{equation*}

The additive noise $\frac{dW(t)}{dt}$ is expressed by \cite{DuNumerical2002}
\begin{equation}\label{SDE1.2}
\frac{dW(t)}{dt} = \sum_{j=1}^{\infty} \sigma_{j} (t) \dot{\beta}_{j} (t) \varphi_{j},
\end{equation}
where $\beta_{j}(t) $ are the independently identically distributed Brownian motions and $\varphi_{j}$, $j=1, 2, \ldots$ are the eigenfunctions of the operator $A$.
Here $\sigma_{j} (t)$, $j=1, 2, \ldots$ are the rapidly decay functions as $j$ increases.
There are two important cases, see \cite{YanGalerkin2005}

Case 1: when $\sigma_{j} (t)=1$, the noise \eqref{SDE1.2} is called the white noise.

Case 2: when $\sigma_{j}(t) = \gamma_{j}^{1/2}$  with $\sum_{j} \gamma_{j} < \infty$, the noise \eqref{SDE1.2} is called trace class noise.
We shall focus on the Case 2 in this work. 

It is well known that the operator $A$ satisfies the resolvent estimate \cite{LubichNonsmooth1996,ThomeeGalerkin2006}
\begin{equation*}\label{resolvent estimate}
\big\| {(z-A)}^{-1} \big\| \leq c_{\theta} |z|^{-1} \quad \forall z\in \Sigma_{\theta}
\end{equation*}
for all $\theta\in (\pi/2, \pi)$, where $\Sigma_{\theta}:=\{ z\in \mathbb{C}\backslash \{0\}:|\arg z| < \theta \}$ is a sector of the complex plane.
Choose the angle $\theta$ such that $\pi/2 < \theta < \min{(\pi, \pi/\alpha)}$, and it holds
\begin{equation}\label{SDE1.4}
\big\| {\left(z^{\alpha}-A\right)}^{-1} \big\| \leq c |z|^{-\alpha} \quad \forall z\in \Sigma_{\theta}.
\end{equation}

The deterministic time fractional diffusion-wave equation \eqref{SDE1.1} provides an effective description of wave propagation, particularly when there is a prevalent power-law relationship between attenuation and frequency.
This behavior is frequently encountered in various scenarios, such as the propagation of acoustic waves in lossy media \cite{SzaboTime1994}.
The exponent of the power-law typically falls within the range of $1$ to $2$, indicating anomalous attenuation \cite{ChenFractional2004}.
In real-world systems, stochastic perturbations emerge from diverse natural sources \cite{Evansintroduction2013}, and their impact cannot be overlooked.
Therefore it is necessary to study the model \eqref{SDE1.1} in both theoretically and numerically.

This paper focuses on the time discretization of the stochastic evolution equation.
Yan \cite{YanGalerkin2005} explored the use of the Euler method for approximating stochastic parabolic equations.
Gunzburger et al. \cite{GunzburgerSharp2019} studied the time discretization methods  for stochastic time-fractional partial differential equations.
Jin et al. \cite{JinNumerical2019} conducted an investigation into the BDF convolution quadrature as an approximation tool for the stochastic subdiffusion equation.
Dai et al. \cite{DaiMittagLeffler2023} considered the Mittag-Leffler Euler integrator for solving the stochastic subdiffusion equation.
Al-Maskari and Karra \cite{AlMaskariStrong2023} focused on the time approximation of the stochastic time-fractional Allen-Cahn equation with a fractional order $ \alpha \in (0, 1)$.
Fahim et al. \cite{FahimSome2023} investigated the approximation of the mild solution of the stochastic fractional order evolution equation.
Kang et al. \cite{KangGalerkin2022} addressed the approximation of the regularized stochastic semilinear subdiffusion equation.
See also \cite{LiGalerkin2022} for the well-posedness analysis and \cite{LiGalerkin2022,GunzburgerConvergence2019,ZouError2018} for the numerical analysis for the stochastic time-fractional models driven by multiplicative noise.

For the model problem \eqref{SDE1.1}, Li et al. \cite{LiGalerkin2017} approximated the noise using a piecewise constant function and derived a regularized problem, demonstrating that the approximate order is $O(\tau)$. Egwu and Yan \cite{egwyan} subsequently extended the results of \cite{LiGalerkin2017} to the semilinear case.
They established that the approximate order is $O(\tau^{\min (1, \alpha + \gamma - \frac{1}{2} - \epsilon)})$, where $ \alpha \in (1, 2)$, $\gamma \in [0, 1]$, and $ \epsilon >0$.
However both works by Li et al. \cite{LiGalerkin2017} and Egwu and Yan \cite{egwyan} did not introduce or analyze any time discretization schemes for \eqref{SDE1.1}.
To the best of our knowledge, no time discretization schemes  were reported for solving the stochastic model \eqref{SDE1.1}.
Comparing with its deterministic counterpart \cite{JinZhou:book}, the development of time discretization schemes for stochastic model \eqref{SDE1.1} is more challenging, due to the additional restriction on the smoothing property caused by the noise.
We apply the convolution quadrature generated by high-order BDF scheme \cite{ChenDiscretized2015,LubichDiscretized1986} to formulate the discretization schemes for \eqref{SDE1.1}.
In order to improve the low convergence rate due to the insufficient regularity, we regularize the additive noise \eqref{SDE1.2} by using an $m$-fold integral-differential calculus.
This novel method, which we refer to as the ID$m$-BDF$k$ method, is able to achieve higher-order convergence in solving the stochastic model \eqref{SDE1.1}.
Such idea has been previously used for dealing with the time discretization schemes for the deterministic subdiffusion with a singular source term $t^{-\gamma}  f(t)$, $\gamma >0$  \cite{ShiHighorder2023}.
The singular source term in \cite{ShiHighorder2023} is  regularized by an $m$-fold integral-differential operator
\begin{equation*}
t^{-\gamma} f(t) = \partial_{t}^{m} \partial_{t}^{-m}\left(t^{-\gamma} f(t)\right)=\partial_{t}^{m} \left( \frac{t^{m-1} }{\Gamma(m)} \ast  \left(t^{-\gamma} f(t)\right) \right),
\end{equation*}
where $v \ast w$ denotes the convolution of $v$ and $w$.

In this work, the additive noise \eqref{SDE1.2} is subjected to a two-step process: first, integration, and then second differentiation, using an $m$-fold integral-differential operator.
We then have, noting  $ \partial_{t}^{-m} \frac{dW (t)}{dt} \big|_{t=0} =0$, $m\geq 1$
\begin{equation}\label{SDE1.5}
\partial_{t}^{-\gamma} \frac{d W(t)}{dt}
= \partial_{t}^{-\gamma}  \partial_{t}^{m} \left( \partial_{t}^{-m}   \frac{d W(t)}{ dt} \right)
= \partial_{t}^{m- \gamma} \left( \frac{t^{m-1} }{\Gamma(m)} \ast   \frac{d W(t)}{ dt} \right),
\end{equation}
which is equivalent to
\begin{equation*}
\partial_{t}^{-\gamma} \frac{d W(t)}{dt}
= \partial_{t}^{m-\gamma}  \partial_{t}^{\gamma-m} \left( \partial_{t}^{-\gamma}   \frac{d W(t)}{dt} \right)
= \partial_{t}^{m- \gamma} \left( \frac{t^{m-1} }{\Gamma(m)} \ast   \frac{d W(t)}{ dt} \right).
\end{equation*}

Let us introduce the main results in this work below.
Let $V(t)= u(t)- \upsilon - tb$.
Then \eqref{SDE1.1} can be written as, with $V(0) = V^{\prime} (0) =0$,
\begin{equation}\label{SDE1.6}
\partial_{t}^{\alpha} V(t) - A  V (t) = A \upsilon + t A b + \partial_{t}^{-\gamma} \frac{dW(t)}{dt}.
\end{equation}

Choose $m=2$ in \eqref{SDE1.5} and denote $g(t) =t \ast \frac{d W(t)}{ dt},$
 we have by \eqref{SDE1.6}, with $V(0) = V^{\prime} (0) =0$,
\begin{equation}\label{SDE1.7}
\partial^{\alpha}_{t} V(t) - A V(t) = \partial_{t} \left( t A \upsilon + \frac{t^2}{2} A b \right) + \partial^{ 2-\gamma}_{t} g(t).
\end{equation}

For the sake of simplicity in notation, this introduction section will only present the key findings concerning the noise term $W$.
Therefore, we currently assume $\upsilon=b=0$.
Applying the Laplace transform method, the solution to the problem \eqref{SDE1.7} could be represented as
\begin{equation}\label{SDE1.8}
V(t) = \big( \mathscr{E} \ast g \big) (t), \quad \text{with} \quad \mathscr{E}(t) = \frac{1}{2\pi i} \int_{\Gamma_{\theta,\kappa}} e^{zt}(z^{\alpha} - A)^{-1} z^{2-\gamma} dz.
\end{equation}
Here $\Gamma_{\theta, \kappa}$ denotes the contour
\begin{equation}\label{SDE1.9}
\Gamma_{\theta, \kappa}= \{ z\in \mathbb{C}: |z|=\kappa, |\arg z|\leq \theta \} \cup \{ z\in \mathbb{C}: z=re^{\pm i\theta}, r\geq \kappa\},
\end{equation}
oriented counterclockwise, with $\pi/2 < \theta < \min{(\pi, \pi/\alpha)}$ and $\kappa>0$.

Let ${\{\sigma_{j}^{l}\}}_{j=1}^{\infty}$ approach ${\{\sigma_{j}\}}_{j=1}^{\infty}$ as $l \rightarrow \infty$ in some appropriate sense \cite{DuNumerical2002}.
Thus an approximation of additive noise $\frac{d W(t)}{ dt} $ in \eqref{SDE1.2} is defined  by
\begin{equation} \label{W_l}
\frac{d W_{l}(t)}{dt} = \sum_{j=1}^{\infty} \sigma_{j}^{l}(t) \dot{\beta}_{j}(t) \varphi_{j}.
\end{equation}
Denoting 
\begin{equation} \label{g_l}
g_{l}(t) 
=t \ast \frac{d W_{l}(t)}{ dt}, 
\end{equation}
and substituting $\frac{d W_{l}(t)}{dt}$ for $\frac{d W(t)}{dt} $ in \eqref{SDE1.7}, we obtain, with $V_{l}(0)= 0$ and $V_{l}'(0)=0$,
\begin{equation} \label{SSDE1.12}
\partial^{\alpha}_{t} V_{l}(t) - A V_{l}(t) = \partial_{t} \left( t A\upsilon + \frac{t^{2}}{2} Ab\right) +  \partial^{2-\gamma}_{t} g_{l}(t),
\end{equation}
which has the solution representation, with $g_{l}(t)$ defined by \eqref{g_l}, 
\begin{equation}\label{SDE1.8_1}
V_{l}(t) = \big( \mathscr{E} \ast g_{l} \big) (t). 
\end{equation}

Let $0 = t_{0} < t_{1} < \dots < t_{N}= T$ be a partition of $[0, T]$ and $\tau$ the time step size.
Let $V_l^{n} \approx V_{l}(t_{n})$ be the approximate solution of $V_{l}(t_{n})$.
Approximating $\partial_{t}^{\beta} V_{l}(t_{n})$, $\beta = \alpha, 2- \gamma$ with the ID$2$-BDF$2$ convolution quadrature, we obtain the following time discretization scheme of \eqref{SDE1.7}, with $ \upsilon = b=0$,
\begin{equation*}
\begin{split}
&{\partial}_{\tau}^{\alpha} V_{l}^{n} - A V_{l}^{n} = {\partial}_{\tau}^{2- \gamma} g_{l} (t_{n}), \quad n=1, 2, \dots,   \\
& V_{l}^{0}=0.
\end{split}
\end{equation*}
Here, the discrete fractional-order derivative $\partial_{\tau}^{\beta}$ is defined by, with $\beta = \alpha, 2- \gamma$,
\begin{equation}\label{SDE1.10}
{\partial}_{\tau}^{\beta} \varphi^{n} = \tau^{- \beta} \sum_{j=0}^{n} w_{j}^{(\beta)} \varphi^{n-j},
\end{equation}
and the weights $w_{j}^{( \beta)}$ are generated by
$\delta_{\tau}^{\beta} (\xi) = \sum_{j=0}^{\infty} w_{j}^{(\beta)} \xi^{j}$,
where
\begin{equation} \label{SDE1.11}
\delta_{\tau}(\xi) := \tau^{-1} \left( \frac{3}{2} - 2 \xi + \frac{1}{2} \xi^{2} \right).
\end{equation}

Using the discrete Laplace transform method, we arrive at
\begin{equation*}
V_{l}^{n} = \left( \mathscr{E}_{\tau} \ast g_{l} \right)(t_{n}),
\end{equation*}
where $\mathscr{E}_{\tau} (t)= \sum_{n=0}^{\infty} \mathscr{E}_{\tau}^{n} \delta (t- t_{n})$.
Here $\mathscr{E}_{\tau}^{n}$ is some approximation of $\mathscr{E} (t_{n})$ and  $ \delta$ means the Dirac $\delta$ function.

In Theorem \ref{SDETheorem2.8}, for any $ \varepsilon >0$, we prove that, with $v=b=0$,
\begin{align}
\mathbb{E} \left\|V_{l}(t_{n})-V_{l}^{n}\right\|^{2} =
\begin{cases}
 O \left( \tau^{2\alpha+2\gamma-1-2\varepsilon}  \right) t_{n}^{2\varepsilon}, &  1 < \alpha+\gamma <  \frac{5}{2}, \notag \\
 O \left( \tau^{4} \right) t_{n}^{2\alpha +2\gamma-5},                         &  \frac{5}{2} < \alpha+\gamma <  3.
\end{cases}
\end{align}

The argument could be easily extended to higher-order schemes.
Denoting 
\begin{equation} \label{g_l_1}
g_{l}(t) 
=\frac{t^2}{2} \ast \frac{d W_{l}(t)}{dt},
\end{equation}
 we may rewrite \eqref{SSDE1.12} as, with $V_{l}(0) = V_{l}^{\prime} (0) =0$,
\begin{equation}\label{SDE1.12}
\partial^{\alpha}_{t} V_{l} (t) - A V_{l}(t) = \partial^{2}_{t} \left( \frac{t^{2}}{2} A\upsilon + \frac{t^{3}}{6} Ab \right) + \partial^{ 3-\gamma}_{t} g_{l} (t).
\end{equation}

Using the Laplace transform method, we have, with $\upsilon = b=0$,
\begin{equation*} \label{exact_1}
V_{l}(t)= \left( \mathscr{E} \ast  g_{l} \right)(t) \quad {\rm and} \quad
\mathscr{E}(t)= \frac{1}{2\pi i} \int_{\Gamma_{\theta,\kappa}} e^{zt}(z^{\alpha} - A)^{-1} z^{3-\gamma} dz,
\end{equation*}
where $\Gamma_{\theta,\kappa}$ is defined in \eqref{SDE1.9}.

Approximating $\partial_{t}^{\beta} \varphi(t)$, $\beta=\alpha, 3-\gamma$ with BDF$3$ convolution quadratures, we obtain the following ID$3$-BDF$3$ time discretization scheme of \eqref{SDE1.12}, with $ \upsilon = b=0$,
\begin{equation*}
\begin{aligned}
&{\partial}_{\tau}^{\alpha} V_{l}^{n} - A V_{l}^{n} = {\partial}_{\tau}^{3 - \gamma} g_{l} (t_{n}), \quad n=1, 2, \dots,   \\
& V_{l}^{0}=0,
\end{aligned}
\end{equation*}
where, with $\beta = \alpha, \, 3- \gamma$,
\begin{equation}\label{SDE1.13}
{\partial}_{\tau}^{\beta} V_{l}^{n} = \tau^{- \beta} \sum_{j=0}^{n} w_{j}^{( \beta)} \xi^{j},
\end{equation}
and the weights $w_{j}^{( \beta)}$ are generated by
$\delta_{\tau}^{\beta}(\xi) = \sum_{j=0}^{\infty} w_{j}^{(\beta)} \xi^{j}$,
where
\begin{equation}\label{SDE1.14}
\delta_{\tau}(\xi) := \tau^{-1} \left( \frac{11}{6} - 3 \xi + \frac{3}{2} \xi^{2} - \frac{1}{3} \xi^3 \right).
\end{equation}

In Theorem \ref{SDETheorem3.5}, for any $ \varepsilon >0$, we prove that, with $\upsilon=b=0$ and for $\alpha \in (1, 2)$  and $\gamma \in (0, 1)$,
\begin{align}
\mathbb{E} \left\| V_{l} (t_{n})- V_{l}^{n}\right\|^{2} = O \left( \tau^{2\alpha+2\gamma-1-2\varepsilon}  \right)   t_{n}^{2\varepsilon},  \notag
\end{align}
which is better than the convergence order $O( \tau^{\min (1, \alpha +\gamma -\frac{1}{2} - \epsilon)})$ obtained in \cite{egwyan} in the case of  $ \alpha +\gamma -\frac{1}{2}  >1$.

The rest of the paper is organized as follows.
In Sections \ref{Sec2} and \ref{Sec3}, we propose the ID$2$-BDF$2$ and ID$3$-BDF$3$ time discretization schemes, respectively, for approximating the stochastic time fractional diffusion-wave equation \eqref{SDE1.1}, where the detailed theoretical analyses of the convergence are derived.
We provide some numerical simulations to show that the numerical results are consistent with the theoretical findings in Section \ref{Sec4}.

Throughout this paper, we denote $c$ as a generic constant that is independent of the step size $\tau$, which could be different at different occurrences. Additionally, we always assume $\epsilon > 0$ is a small positive constant.

\section{ID$2$-BDF$2$ method}\label{Sec2}
For the linear stochastic fractional evolution equation with integrated additive noise, the predominant time-stepping methods lead to  low order error estimates with $O\left(\tau^{\min\{\alpha + \gamma -1/2,1\}}\right)$, see \cite{egwyan,KangGalerkin2022}, also see Table \ref{table:1} by ID$1$-BDF$2$ method.
To break the first-order barrier in such time-stepping methods, it motivates us to consider the ID$2$-BDF2 method.
In this work, we shall frequently use the following It\^{o} isometry property, see \cite[p.67]{Evansintroduction2013}.
\begin{proposition}[It\^{o} Isometry Property]\label{itoIs}
Let $ \{ \psi (s): s \in [0, T] \}$ be a real-valued predictable process such that   $  \int^{T}_{0} \mathbb{E} | \psi (s) |^2 ds < \infty$.  Let $B(t)$ denote a real-valued standard Brownian motion. Then the following isometry equality holds for $t \in (0, T]$
\begin{equation*}
\mathbb{E} \left\| \int_{0}^{t} \psi(s) d B(s) \right\|^{2} = \int_{0}^{t} \mathbb{E} \left\|\psi(s)\right\|^{2} ds,
\end{equation*}
where $\mathbb{E}$ denotes the expectation.
\end{proposition}

\subsection{Numerical scheme and solution representation}\label{Sec2.1}
Taking the Laplace transform in both sides of~\eqref{SDE1.7}, it leads to
\begin{equation*}
\widehat{V}(z)= {(z^{\alpha}-A)}^{-1}\left( z^{-1}A\upsilon  + z^{-2}Ab  + z^{2 -\gamma} \widehat{g}  \right).
\end{equation*}
By the inverse Laplace transform, there exists~\cite{GunzburgerSharp2019,JinCorrection2017,ShiHighorder2023}
\begin{equation}\label{SDE2.1}
V(t) = \frac{1}{2\pi i} \int_{\Gamma_{\theta, \kappa}} e^{zt} {(z^{\alpha}-A)}^{-1} \left( z^{-1}A\upsilon + z^{-2}Ab + z^{2-\gamma} \widehat{g} \right) dz,
\end{equation}
where  $\Gamma_{\theta, \kappa}$ is defined in \eqref{SDE1.9}.

The solution~\eqref{SDE2.1} is also represented as
\begin{equation}\label{SDE2.2}
\begin{aligned}
V(t)
& = E(1,t) A\upsilon  + E(2,t) Ab +  \left(  \mathscr{E}  \ast  g  \right)  (t)     \\
& = E(1,t) A\upsilon  + E(2,t) Ab +  \left(  E (\gamma, t)  \ast \frac{d W(t)}{dt}   \right)(t),
\end{aligned}
\end{equation}
where $\mathscr{E}$ is defined by \eqref{SDE1.8} and, with $q=1, 2, \gamma$,
\begin{equation}\label{SDE2.3}
E(q,t) \varphi = \frac{1}{2\pi i} \int_{\Gamma_{\theta, \kappa}} e^{zt}  {(z^{\alpha}-A)}^{-1}  z^{-q} \varphi dz.
\end{equation}

 Substituting $\frac{d W_{l}(t)}{dt}$ for $\frac{d W(t)}{dt} $ in \eqref{SDE1.7}, we obtain, for $g_{l}(t)$ defined by \eqref{g_l} and 
with $V_{l}(0)= 0$ and $V_{l}'(0)=0$, 
\begin{equation} \label{SDE2.4}
\partial^{\alpha}_{t} V_{l}(t) - A V_{l}(t) = \partial_{t} \left( t A\upsilon + \frac{t^{2}}{2} Ab\right) +  \partial^{2-\gamma}_{t} g_{l}(t),
\end{equation}
which has the solution representation
\begin{equation}\label{SDE2.5}
\begin{aligned}
V_{l} (t)
& = E(1,t) A\upsilon + E(2,t) Ab + \left( \mathscr{E} \ast g_{l} \right)(t)  \\
& = E(1,t) A\upsilon + E(2,t) Ab + \left(  E(\gamma, t) \ast \frac{d W_{l}(t)}{dt} \right)(t). 
\end{aligned}
\end{equation}
The following theorem shows that $V_{l}(t)$ in \eqref{SDE2.5} indeed approximates $V(t)$ in \eqref{SDE2.2}.
\begin{theorem}\label{SDETheorem2.2}
Let $\left\{\sigma_{j}\left( t \right) \right\}$ be uniformly bounded by
\begin{equation*}
\left| \sigma_{j}\left( t \right) \right| \leq \mu_{j} \quad \forall t\in [0,T],
\end{equation*}
and the coefficients $\left\{ \sigma _{j}^{l} \right\}$ are constructed such that
\begin{equation*}
\left| \sigma_{j}\left( t \right) - \sigma_{j}^{l} \left( t \right)  \right| \leq \eta _{j}^{l}, \ \ \left| \sigma_{j}^{l}\left( t \right) \right| \leq \mu_{j}^{l} \ \ \forall t\in [0,T]
\end{equation*}
with positive sequences $\left\{ \eta _{j}^{l} \right\}$ being arbitrarily chosen, and $\left\{ \mu_{j}^{l} \right\}$  being related to $\left\{ \eta_{j}^{l} \right\}$ and $\left\{ \mu_{j} \right\}$. Assume that the infinite series
$\sum_{j=1}^{\infty} {(\eta_{j}^{l})}^{2}$ and $\sum_{j=1}^{\infty}\left( \mu_{j}^{l} \right)^{2} $
are convergent. Let $V_{l}(t)$ and $V(t)$ be the solutions of~\eqref{SDE2.5} and~\eqref{SDE2.2} with  $\alpha+\gamma>\frac{1}{2}$, respectively.
Then, for some constant $c > 0$, it holds
\begin{equation*}
\mathbb{E} \| V(t) - V_{l}(t) \|^{2} \leq c \sum_{j=1}^{\infty} {(\eta_{j}^{l})}^{2}.
\end{equation*}
\end{theorem}

\begin{proof}
Subtracting~\eqref{SDE2.5} from~\eqref{SDE2.2}, it yields
\begin{equation*}
V(t) - V_{l}(t) = E(\gamma,t) \ast \left(  \frac{d W(t)}{dt} -  \frac{d W_{l}(t)}{dt} \right).
\end{equation*}
Using the orthogonality of $\varphi_{j}$ and It\^{o} isometry property, we obtain
\begin{equation*}
\begin{split}
\mathbb{E} \| V(t) - V_{l}(t)\|^{2}
& = \mathbb{E} \left\| \int_{0}^{t} E(\gamma,t-s) \sum_{j=1}^{\infty} \left( \sigma_{j}(s) - \sigma_{j}^{l}(s) \right) \varphi_{j} d \beta_{j}(s)\right\|^{2}\\
& = \sum_{j=1}^{\infty} \int_{0}^{t} \left\| E(\gamma,t-s) \right\|^{2} {\left( \sigma_{j}(s) - \sigma_{j}^{l}(s) \right)}^{2} ds\\
& \leq \sum_{j=1}^{\infty} {\left( \eta_{j}^{l} \right)}^{2} \int_{0}^{t} \left\| E(\gamma,s) \right\|^{2} ds.
\end{split}
\end{equation*}
According to~\eqref{SDE2.3} and \eqref{SDE1.4}, choose $\beta \in (0,1)$, we estimate
\begin{equation*}\label{eq2.16}
\begin{split}
\int_{0}^{t} \left\|  E(\gamma,s) \right\|^{2} ds
& \leq c \int_{0}^{t} \left( \int_{\Gamma_{\theta, \kappa}} \left| e^{zs} \right| |z|^{-\alpha-\gamma} |dz| \right)^{2} ds \\
& \leq c \int_{0}^{t} \int_{\Gamma_{\theta, \kappa}} \left| e^{zs} \right| |z|^{-\beta} |dz| \int_{\Gamma_{\theta, \kappa}} \left| e^{zs} \right| |z|^{-2\alpha-2\gamma+\beta} |dz| ds \\
& \leq c \int_{0}^{t} s^{\beta-1} \int_{\Gamma_{\theta, \kappa}} \left| e^{zs} \right| |z|^{-2\alpha-2\gamma+\beta} |dz| ds
\leq c \kappa^{-2\alpha-2\gamma+1}.
\end{split}
\end{equation*}
Here, for $\alpha+\gamma>\frac{1}{2}$, we use
\begin{equation*}
\begin{split}
& \int_{0}^{t} s^{\beta-1} \int_{\Gamma_{\theta, \kappa}} \left| e^{zs} \right| |z|^{-2\alpha-2\gamma+\beta} |dz| ds \\
& = \int_{0}^{t}\left( s^{\beta-1} \int_{\kappa}^{+\infty} e^{rs\cos \theta} r^{-2\alpha-2\gamma+\beta} dr
    +  s^{\beta-1} \int_{-\theta}^{\theta} e^{\kappa s \cos \varphi} \kappa^{-2\alpha-2\gamma+1+\beta} d\varphi\right) ds \\
& = \int_{\kappa}^{+\infty} r^{-2\alpha-2\gamma+\beta} dr \int_{0}^{t} s^{\beta-1}  e^{rs\cos \theta} ds
    +  \int_{-\theta}^{\theta} \kappa^{-2\alpha-2\gamma+1+\beta} d\varphi \int_{0}^{t} s^{\beta-1}  e^{\kappa s \cos \varphi}  ds \\
& \leq c \int_{\kappa}^{+\infty} r^{-2\alpha-2\gamma} dr  + c \int_{-\theta}^{\theta}  \kappa^{-2\alpha-2\gamma+1} d\varphi
\leq c \kappa^{-2\alpha-2\gamma+1}.
\end{split}
\end{equation*}
Hence, for $\alpha+\gamma>\frac{1}{2}$, we have
\begin{equation*}
\mathbb{E} \| V(t) - V_{l}(t) \|^{2} \leq  c \sum_{j=1}^{\infty} {\left( \eta_{j}^{l} \right)}^{2} \kappa^{-2\alpha-2\gamma+1} \leq c \sum_{j=1}^{\infty} {\left( \eta_{j}^{l} \right)}^{2}.
\end{equation*}
The proof is completed.
\end{proof}
\begin{remark}\label{SDERemark2.3}
Note that it does't require to be uniformly bounded for the derivative of $ \sigma _{j}\left( t \right)$ in Theorem \ref{SDETheorem2.2},
which weaken the conditions of Theorem 3.3 in \cite{DuNumerical2002} and Theorem 2.6 in \cite{LiGalerkin2017}.
\end{remark}

Let $V_{l}^{n} \approx V_{l} (t_{n})$ be the approximate solution of $V_{l}(t_{n})$ in \eqref{SDE2.4}. Approximating $\partial_{t}^{ \beta} \varphi (t_{n})$, $\beta =1, \alpha, 2- \gamma$ with the BDF$2$ convolution quadrature, we obtain the following ID$2$-BDF$2$ time discretization scheme of \eqref{SDE2.4}, with $V_{l}^{0}=0$,  
\begin{equation}\label{SDE2.6}
{\partial}_{\tau}^{\alpha} V_{l}^{n} - A V_{l}^{n} = {\partial}_{\tau} \left( t_{n} A\upsilon + \frac{t_{n}^{2}}{2} Ab \right) + {\partial}_{\tau}^{2- \gamma} g_{l} (t_{n}), \quad n=1, 2, \ldots,  N,
\end{equation}
where ${\partial}_{\tau}^{\beta} \varphi^{n}$, $\beta=1, \alpha, 2- \gamma$ are defined in \eqref{SDE1.10}.

Given a sequence $(\kappa_n)_0^\infty$ and take $\widetilde{\kappa}(\zeta)=\sum_{n=0}^{\infty}\kappa_n \zeta^n$ to be its generating power series.
The representation of the discrete solution in \eqref{SDE2.6} is obtained by the following.
\begin{lemma}\label{SDELemma2.4}
Let $\delta_{\tau}$ be given in \eqref{SDE1.11} and $\rho_{j}(\xi)=\sum^{\infty}_{n=1}n^{j} \xi^{n}$ with $j=1, 2$.
Let $g_{l}(t)$ be defined by \eqref{g_l}. 
Then the discrete solution of \eqref{SDE2.6} is represented by
\begin{align}
V_{l}^{n}
& =\frac{1}{2\pi i}\int_{\Gamma^{\tau}_{\theta,\kappa}} e^{zt_n} {(\delta^{\alpha}_{\tau}(e^{-z\tau})-A)}^{-1} \delta_{\tau}(e^{-z\tau}) \tau \left( \rho_{1}(e^{-z\tau}) \tau A \upsilon +  \frac{\rho_{2}(e^{-z\tau})}{2} \tau^{2} Ab  \right) dz \notag  \\
& \quad + \frac{1}{2\pi i}\int_{\Gamma^{\tau}_{\theta,\kappa}} e^{zt_n} {(\delta^{\alpha}_{\tau}(e^{-z\tau})-A)}^{-1} \delta^{2-\gamma}_{\tau}(e^{-z\tau}) \tau  \widetilde{g_{l}}(e^{-z\tau})  dz \label{SDE2.7}
\end{align}
with $\Gamma^{\tau}_{\theta, \kappa}=\{z\in \Gamma_{\theta, \kappa}: |\Im z|\leq \pi / \tau\}$ and $\widetilde{g_{l}} (e^{-z\tau}) = \widetilde{ t_{n} \ast \frac{\partial W_{l}(t_{n})}{\partial t}} (e^{-z\tau})$.
\end{lemma}
\begin{proof}
Multiplying the~\eqref{SDE2.6} by $\xi^{n}$ and summing over $n$ with $V_{l}^0=0$, we obtain
\begin{equation*}
\sum^{\infty}_{n=1} \partial^{\alpha}_{\tau} V_{l}^{n} \xi^{n} - \sum^{\infty}_{n=1}  AV_{l}^{n}  \xi^{n}
= \sum^{\infty}_{n=1} {\partial}_{\tau} \left( t_{n} A\upsilon + \frac{t_{n}^{2}}{2} Ab\right)  + \sum^{\infty}_{n=1} \partial^{2-\gamma}_{\tau}  g_{l}(t_{n})   \xi^{n}.
\end{equation*}
Note that
\begin{equation*}
\begin{split}
\sum^{\infty}_{n=1} \partial^{\alpha}_{\tau} V_{l}^{n} \xi^{n}
=& \sum^{\infty}_{n=1} \frac{1}{\tau^{\alpha}}\sum^{n}_{j=0} \omega^{(\alpha)}_{j} V_{l}^{n-j} \xi^{n}
=  \sum^{\infty}_{j=0} \frac{1}{\tau^{\alpha}}\sum^{\infty}_{n=j} \omega^{(\alpha)}_{j} V_{l}^{n-j} \xi^{n}\\
=& \sum^{\infty}_{j=0} \frac{1}{\tau^{\alpha}}\sum^{\infty}_{n=0} \omega^{(\alpha)}_{j} V_{l}^{n} \xi^{n+j}
=  \frac{1}{\tau^{\alpha}} \sum^{\infty}_{j=0} \omega^{(\alpha)}_{j} \xi^{j} \sum^{\infty}_{n=0}  V_{l}^{n} \xi^{n}
=  \delta^{\alpha}_{\tau} (\xi) \widetilde{V_{l}} (\xi).
\end{split}
\end{equation*}
Similarly, one has
\begin{equation*}
\sum^{\infty}_{n=1} {\partial}_{\tau} t_{n} A\upsilon  \xi^{n} = \delta_{\tau}(\xi) \rho_{1}(\xi) \tau A\upsilon, \quad
\sum^{\infty}_{n=1} {\partial}_{\tau} t^{2}_{n} Ab  \xi^{n} = \delta_{\tau}(\xi) \rho_{2}(\xi) \tau^{2} Ab
\end{equation*}
with $\rho_{j}(\xi)=\sum^{\infty}_{n=1}n^{j} \xi^{n}$, $j=1, 2$ and
\begin{equation*}
\sum^{\infty}_{n=1} \partial^{2-\gamma}_{\tau} g_{l}(t_{n}) \xi^{n}
 = \sum^{\infty}_{n=1} \partial^{2-\gamma}_{\tau}  g_{l}^{n} \xi^{n} = \delta^{2-\gamma}_{\tau}(\xi) \widetilde{g_{l}} (\xi). 
\end{equation*}
It leads to
\begin{equation}\label{SDE2.8}
\begin{aligned}
\widetilde{V_{l}}(\xi) 
& = {\left( \delta^{\alpha}_{\tau}(\xi)- A\right)}^{-1} \delta_{\tau}(\xi) \left(\rho_{1}(\xi) \tau A\upsilon +\frac{\rho_{2}(\xi)}{2} \tau^{2} Ab\right)\\
& \quad + {\left( \delta^{\alpha}_{\tau}(\xi) -A\right)}^{-1} \delta^{2-\gamma}_{\tau}(\xi) \widetilde{g_{l}} (\xi) .
\end{aligned}
\end{equation}
According to Cauchy's integral formula, and the change of variables $\xi=e^{-z\tau}$, and Cauchy's theorem, one has~\cite{JinCorrection2017}
\begin{equation*}\label{DLT}
\begin{split}
V_{l}^{n}
& =\frac{1}{2\pi i}\int_{\Gamma^{\tau}_{\theta,\kappa}} e^{zt_n} {(\delta^{\alpha}_{\tau}(e^{-z\tau})-A)}^{-1} \delta_{\tau}(e^{-z\tau}) \tau \left( \rho_{1}(e^{-z\tau}) \tau A \upsilon +  \frac{\rho_{2}(e^{-z\tau})}{2} \tau^{2} Ab  \right) dz \\
& \quad + \frac{1}{2\pi i}\int_{\Gamma^{\tau}_{\theta,\kappa}} e^{zt_n} {(\delta^{\alpha}_{\tau}(e^{-z\tau})-A)}^{-1} \delta^{2-\gamma}_{\tau}(e^{-z\tau}) \tau  \widetilde{g_{l}}(e^{-z\tau}) dz
\end{split}
\end{equation*}
with $\Gamma^{\tau}_{\theta, \kappa}=\{z\in \Gamma_{\theta, \kappa}: |\Im z|\leq \pi / \tau\}$. 
The proof is completed.
\end{proof}

\subsection{Error estimates}\label{Sec2.2}
We first give some lemmas which will be used in the error estimates.
\begin{lemma}\cite{JinCorrection2017}\label{SDELemma2.5}
Let $\delta_{\tau}(\xi)$ be given in~\eqref{SDE1.11}. Then there exist the positive constants $c_{1},c_{2}$, $c$ and $\theta \in (\pi/2, \theta_{\varepsilon})$ with $\theta_{\varepsilon} \in (\pi/2, \pi),~\forall \varepsilon>0$ such that
\begin{equation*}
\begin{split}
& c_{1}|z|\leq |\delta_{\tau}(e^{-z\tau})| \leq c_{2}|z|, \quad |\delta_{\tau}(e^{-z\tau})-z|\leq c \tau^{2}|z|^{3},\\
& |\delta^{\alpha}_{\tau}(e^{-z\tau})-z^{\alpha}|\leq c \tau^{2}|z|^{2+\alpha}, \quad \delta_{\tau}(e^{-z\tau}) \in \Sigma_{\pi/2+\varepsilon} ~ {\forall} z\in \Gamma^{\tau}_{\theta,\kappa}.
\end{split}
\end{equation*}
\end{lemma}

\begin{lemma}\label{SDELemma2.6}
Let $\delta_{\tau}(\xi)$ be given in~\eqref{SDE1.11} and $\rho_{1}(\xi)=\sum^{\infty}_{n=1} n \xi^{n}$.
Then there exists a positive constants $c$ such that
\begin{equation*}\label{eq2.12}
\left|  \rho_{1}(e^{-z\tau})  \tau^{2} - z^{-2} \right| \leq c \tau^{2}~~{\rm and}~~
\left|\delta^{2-\gamma}_{\tau}(e^{-z\tau})  \rho_{1}(e^{-z\tau}) \tau^{2} - z^{-\gamma} \right| \leq c \tau^{2}|z|^{2-\gamma}~~ \forall z\in \Gamma^{\tau}_{\theta,\kappa},
\end{equation*}
where $\theta\in (\pi/2,\pi)$ is sufficiently close to $\pi/2$.
\end{lemma}
\begin{proof}
For a rigorous proof of the first estimate the reader is referred to Lemma 3.2 in  \cite{ShiHighorder2023}.
Here, we mainly prove the second one.  Let
\begin{equation*}
\delta^{2-\gamma}_{\tau}(e^{-z\tau}) \rho_{1}(e^{-z\tau})  \tau^{2} - z^{-\gamma}= J_{1} +J_{2}
\end{equation*}
with
\begin{equation*}
J_{1}=\delta^{2-\gamma}_{\tau}(e^{-z\tau}) \rho_{1}(e^{-z\tau}) \tau^{2}-\delta^{2-\gamma}_{\tau}(e^{-z\tau})z^{-2} \quad {\rm and} \quad J_{2}=\delta^{2-\gamma}_{\tau}(e^{-z\tau})z^{-2}-z^{-\gamma}.
\end{equation*}
According to Lemma \ref{SDELemma2.5}, we have
\begin{equation*}
\left| J_{1} \right| \leq \left| \delta^{2-\gamma}_{\tau}(e^{-z\tau}) \right| \left| \rho_{1}(e^{-z\tau}) \tau^{2}- z^{-2} \right|
\leq c \tau ^{2}\left| z \right|^{2-\gamma}
\end{equation*}
and
\begin{equation*}
\left| J_{2} \right| \leq \left| \delta^{2-\gamma}_{\tau}(e^{-z\tau})-z^{2-\gamma} \right| \left| z^{-2}  \right| \leq c\tau ^{2} \left| z \right|^{2-\gamma }.
\end{equation*}
By the triangle inequality, the desired result is obtained.
\end{proof}

\begin{lemma}\label{SDELemma2.7}
Let $\delta_{\tau}(\xi)$ be given by~\eqref{SDE1.11} and $\rho_{j}(\xi)=\sum^{\infty}_{n=1}n^{j} \xi^{n}$ with $j=1, 2$. Then there exists a positive constants $c$ such that
\begin{equation*}
\left\| {\left(\delta^{\alpha}_{\tau}(e^{-z\tau}) - A\right)}^{-1} \delta_{\tau}(e^{-z\tau}) \frac{\rho_{j}(e^{-z\tau})}{\Gamma(j+1)} \tau^{j+1} A - {(z^{\alpha}-A)}^{-1} z^{-j} A \right\| \leq  c \tau^{2}|z|^{2-j}
\end{equation*}
and
\begin{equation*}
\left\| {\left(\delta^{\alpha}_{\tau}(e^{-z\tau}) - A\right)}^{-1} \delta^{2-\gamma}_{\tau}(e^{-z\tau}) \rho_{1}(e^{-z\tau}) \tau^{2} - {(z^{\alpha}-A)}^{-1} z^{-\gamma} \right\|
\leq  c \tau^{2}|z|^{2-\alpha-\gamma}.
\end{equation*}
\end{lemma}
\begin{proof}
The first inequality of this lemma can be performed similarly as Lemmas 3.5 and 3.6 in  \cite{Shi2023highorder}. For the second one, let
\begin{equation*}
{\left( \delta^{\alpha}_{\tau}(e^{-z\tau}) -A\right)}^{-1} \delta^{2-\gamma}_{\tau}(e^{-z\tau})  \rho_{1}(e^{-z\tau})  \tau^{2} - {(z^{\alpha}-A)}^{-1} z^{-\gamma}=I+II
\end{equation*}
with
\begin{equation*}
\begin{split}
I=& {\left( \delta^{\alpha}_{\tau}(e^{-z\tau}) -A\right)}^{-1}[\delta^{2-\gamma}_{\tau}(e^{-z\tau}) \rho_{1}(e^{-z\tau}) \tau^{2}-z^{-\gamma } ], \\
II=& [{\left( \delta^{\alpha}_{\tau}(e^{-z\tau}) -A\right)}^{-1}- {\left( z^{\alpha } -A \right)}^{-1}  ]z^{-\gamma }.
\end{split}
\end{equation*}
The resolvent estimate~\eqref{SDE1.4} and Lemma \ref{SDELemma2.5} imply directly
\begin{equation}\label{SDE2.9}
\left\| {\left( \delta^{\alpha}_{\tau}(e^{-z\tau}) -A\right)}^{-1} \right\|  \leq \left| z \right|^{-\alpha }.
\end{equation}
From Lemma \ref{SDELemma2.6} and~\eqref{SDE2.9}, we obtain
\begin{equation*}
\left\| I \right\| \leq  c\tau ^{2} \left| z \right|^{2-\alpha -\gamma }.
\end{equation*}
Using Lemma \ref{SDELemma2.5}, \eqref{SDE2.9} and the identity
\begin{equation*}\label{mm3.3}
{\left( \delta^{\alpha}_{\tau}(e^{-z\tau}) -A\right)}^{-1}- {\left( z^{\alpha } -A \right)} ^{-1}=\left( z^{\alpha }- \delta^{\alpha}_{\tau}(e^{-z\tau}) \right) {\left( \delta^{\alpha}_{\tau}(e^{-z\tau})-A \right)}^{-1} {\left( z^{\alpha } -A \right)}^{-1},
\end{equation*}
we estimate II as following
\begin{equation*}
\left\| II \right\| \leq c\tau ^{2}\left| z \right| ^{2+\alpha }c\left| z \right| ^{-\alpha }c\left| z \right|^{-\alpha } \left| z \right| ^{-\gamma }\leq c\tau ^{2}  \left| z \right| ^{2-\alpha -\gamma}.
\end{equation*}
By the triangle inequality, the desired result is obtained.
\end{proof}

We now turn to our main theorems.
\begin{theorem}\label{SDETheorem2.8}
Let $V_{l}(t_{n})$ and $V_{l}^{n}$ be the solutions of~\eqref{SDE2.5} and~\eqref{SDE2.7}, respectively.
Let $\upsilon = b = 0$. 
Then for any small $\varepsilon>0$
\begin{equation*}
\mathbb{E} \left\|V_{l}(t_{n})-V_{l}^{n}\right\|^{2}
\leq c \sum_{j=1}^{\infty}\left( \mu_{j}^{l} \right)^{2}  \tau^{2\alpha+2\gamma-1-2\varepsilon}  t_{n}^{2\varepsilon} \quad {\rm for} ~ 1< \alpha+\gamma \leq \frac{5}{2},
\end{equation*}
and
\begin{equation*}
\mathbb{E} \left\|V_{l}(t_{n})-V_{l}^{n}\right\|^{2}
\leq c  \sum_{j=1}^{\infty }\left( \mu_{j}^{l} \right)^{2}  \tau^{4} t_{n}^{2\alpha +2\gamma-5} \quad  {\rm for} ~ \frac{5}{2} < \alpha+\gamma <  3.
\end{equation*}
\end{theorem}
\begin{proof}
By~\eqref{SDE2.5}, we obtain
\begin{equation}\label{SDE2.10}
\begin{split}
V_{l}(t_{n})
= & \left( E(\gamma,t) \ast \frac{d W_{l}(t)}{dt} \right) \left( t_{n}  \right)
= \left(\left(\mathscr{E}(t)\ast t  \right) \ast  \frac{d W_{l}(t)}{dt}   \right)(t_{n}),
\end{split}
\end{equation}
where $\mathscr{E}(t)$ is defined by \eqref{SDE1.8}.

From~\eqref{SDE2.8} and $g_{l}^{n} = t_{n} \ast \frac{d W_{l} (t)}{dt}$, it yields
\begin{equation*}
\begin{split}
\widetilde{V_{l}}(\xi)
&=\left( \delta^{\alpha}_{\tau}(\xi) -A\right)^{-1} \delta^{2-\gamma}_{\tau}(\xi) \widetilde{g_{l}}(\xi) = \widetilde{\mathscr{E_{\tau}}}(\xi)\widetilde{g_{l}}(\xi)
=\sum^{\infty}_{n=0}\mathscr{E}^{n}_{\tau}\xi^{n}\sum^{\infty}_{j=0} g_{l}^{j} \xi^{j}\\
&=\sum^{\infty}_{n=0}\sum^{\infty}_{j=0}\mathscr{E}^{n}_{\tau} g_{l}^{j} \xi^{n+j}=\sum^{\infty}_{j=0}\sum^{\infty}_{n=j}\mathscr{E}^{n-j}_{\tau} g_{l}^{j} \xi^{n}
=\sum^{\infty}_{n=0}\sum^{n}_{j=0}\mathscr{E}^{n-j}_{\tau} g_{l}^{j} \xi^{n}=\sum^{\infty}_{n=0}V_{l}^n\xi^{n}
\end{split}
\end{equation*}
with
\begin{equation*}
V_{l}^{n}=\sum^{n}_{j=0}\mathscr{E}^{n-j}_{\tau} g_{l}^{j} := \sum^{n}_{j=0}\mathscr{E}^{n-j}_{\tau} g_{l}(t_{j}).
\end{equation*}
Here $\sum^{\infty}_{n=0}\mathscr{E}^{n}_{\tau}\xi^{n}=\widetilde{\mathscr{E_{\tau}}}(\xi)=\left( \delta^{\alpha}_{\tau}(\xi) -A\right)^{-1} \delta^{2-\gamma}_{\tau}(\xi)$.
From  the Cauchy's integral formula and the change of variables $\xi=e^{-z\tau}$, we obtain the representation of the $\mathscr{E}^{n}_{\tau}$ as following 
 \begin{equation*}
\mathscr{E}^{n}_{\tau}=\frac{1}{2\pi i}\int_{|\xi|=\rho}{\xi^{-n-1}\widetilde{\mathscr{E_{\tau}}}(\xi)}d\xi
=\frac{\tau}{2\pi i}\int_{\Gamma^{\tau}_{\theta,\kappa}} {e^{zt_n}\left( \delta^{\alpha}_{\tau}(e^{-z\tau}) -A\right)^{-1} \delta^{2-\gamma}_{\tau}(e^{-z\tau}) }dz,
\end{equation*}
where $\theta\in (\pi/2,\pi)$ is sufficiently close to $\pi/2$ and $\kappa=t_{n}^{-1}$ in~\eqref{SDE1.9}.

According to Lemma \ref{SDELemma2.5}, \eqref{SDE2.9} and $\tau t^{-1}_{n} = \frac{1}{n}\leq 1$, there exists
\begin{equation}\label{SDE2.11}
\begin{aligned}
\|\mathscr{E}^{n}_{\tau}\| 
& \leq c \tau \left( \int^{\frac{\pi}{\tau\sin\theta}}_{\kappa} e^{rt_{n}\cos\theta} r^{2-\alpha-\gamma}dr 
         +\int^{\theta}_{-\theta}e^{\kappa t_{n}\cos\psi} \kappa^{3-\alpha-\gamma}  d\psi\right) \\
& \leq c\tau t_{n}^{\alpha+\gamma-3}.
\end{aligned}
\end{equation}
Let $ \mathscr{E}_{\tau}(t)=\sum^{\infty}_{n=0}\mathscr{E}^{n}_{\tau}\delta_{t_{n}}(t)$, with $\delta_{t_{n}}$ being the Dirac delta function at $t_{n}$.
Then
\begin{equation}\label{SDE2.12}
(\mathscr{E}_{\tau}(t)\ast g_{l}(t))(t_{n})
 = \left(\sum^{\infty}_{j=0}\mathscr{E}^{j}_{\tau}\delta_{t_{j}}(t) \ast g_{l}(t) \right)(t_{n})
 = \sum^{n}_{j=0}\mathscr{E}^{n-j}_{\tau} g_{l}(t_{j})=V^{n}_{l}.
\end{equation}
Moreover, using the above equation, there exists
\begin{equation*}
\begin{split}
  \widetilde{(\mathscr{E}_{\tau}\ast t)}(\xi)
& = \sum^{\infty}_{n=0} \sum^{n}_{j=0}\mathscr{E}^{n-j}_{\tau}t_{j}\xi^{n}  =\sum^{\infty}_{j=0} \sum^{\infty}_{n=j}\mathscr{E}^{n-j}_{\tau}t_{j}\xi^{n}
  =\sum^{\infty}_{j=0} \sum^{\infty}_{n=0}\mathscr{E}^{n}_{\tau}t_{j}\xi^{n+j}\\
& =\sum^{\infty}_{n=0}\mathscr{E}^{n}_{\tau}\xi^{n}\sum^{\infty}_{j=0}t_{j}\xi^{j}  =\widetilde{\mathscr{E_{\tau}}}(\xi) \tau \sum^{\infty}_{j=0}j \xi^{j}
  =\widetilde{\mathscr{E}_{\tau}}(\xi) \tau \rho_{1}(\xi).
\end{split}
\end{equation*}
According to \eqref{SDE1.2}, \eqref{SDE2.10}, \eqref{SDE2.12} and It\^{o} isometry property, we have
\begin{equation*}\label{mm3.8}
\begin{split}
\mathbb{E} \left\|V_{l}(t_{n})-V_{l}^{n}\right\|^{2}
= & \mathbb{E} \left\| \left( \left[ \left( \mathscr{E}-\mathscr{E}_{\tau }  \right)\ast t \right] \ast  \frac{\partial  W_{l} \left( t \right) }{\partial t} \right) \left( t_{n}  \right) \right\|^{2} \\
= & \mathbb{E} \left\| \sum_{j=1}^{\infty }\int_{0}^{t_{n} } \left[ \left( \left( \mathscr{E}-\mathscr{E}_{\tau }  \right)\ast t \right) \left( t_{n}-s  \right) \right]  \sigma _{j}^{l}\left( s \right)  \varphi _{j} d \beta_{j}\left( s \right) \right\|^{2} \\
= & \sum_{j=1}^{\infty } \int_{0}^{t_{n} } \left\| \left[ \left( \left( \mathscr{E}-\mathscr{E}_{\tau } \right)\ast t \right) \left( t_{n}-s \right) \right] \sigma _{j}^{l}\left( s \right) \varphi_{j} \right\|^{2} ds \\
\leq & \sum_{j=1}^{\infty } \left( \mu _{j}^{l}  \right)^{2} \int_{0}^{t_{n} }\left\| \left( \left( \mathscr{E}-\mathscr{E}_{\tau } \right)\ast t  \right) \left( t_{n}-s  \right)  \right\|^{2}ds.
\end{split}
\end{equation*}

Next, we prove the following inequality for any $t\in \left( t_{n-1},t_{n} \right)$ 
\begin{equation*}\label{mm3.9}
\left\| \left( \left( \mathscr{E}-\mathscr{E}_{\tau } \right)\ast t \right) \left( t  \right)  \right\| \leq \left\{
\begin{split}
& c\tau ^{2}t^{\alpha +\gamma -3},  \qquad \quad \ \ \frac{5}{2} < \alpha+\gamma <  3, \\
& c \tau^{\alpha+\gamma-\frac{1}{2}-\varepsilon}t^{-\frac{1}{2}+\varepsilon}, \quad  1< \alpha+\gamma \leq \frac{5}{2}.
\end{split}\right.
\end{equation*}
By Taylor series expansion of $\mathscr{E} \left( t \right) $ at $t=t_{n} $, we get
\begin{equation*}
\begin{split}
\left(\mathscr{E} \left( t \right)\ast t \right)(t)
=  \left(\mathscr{E} \left( t \right)\ast t\right) ( t_{n} ) + \left( t-t_{n} \right) \left(\mathscr{E} \left( t \right)\ast 1\right) ( t_{n} )
 + \int _{t}^{t_{n} } \left( s-t \right)  \mathscr{E}\left( s \right)  ds,
\end{split}
\end{equation*}
which also holds for the convolution $\left( \mathscr{E}_{\tau}\ast t \right)\left( t \right) $.
Now, we estimate 
\begin{equation*}
\begin{split}
\left\| \left( \left( \mathscr{E}-\mathscr{E}_{\tau } \right)\ast t \right) \left( t_{n}  \right)  \right\|
\leq \left\| J_{1}  \right\| +\left\| J_{2}  \right\|
\end{split}
\end{equation*}
with
\begin{equation*}
\begin{split}
J_{1}=&\frac{1}{2\pi i}\int _{\Gamma _{\theta ,k}\setminus \Gamma _{\theta ,k}^{\tau}} e^{zt_{n}}\left( z^{\alpha } -A \right)^{-1}\frac{1}{z^{\gamma } } dz, \\
J_{2}=&\frac{1}{2\pi i}\int _{\Gamma _{\theta ,k}^{\tau}} e^{zt_{n}}\left[ \left( z^{\alpha } -A \right)^{-1} z^{-\gamma} - \left( \delta^{\alpha}_{\tau}(e^{-z\tau}) -A\right)^{-1} \delta^{2-\gamma}_{\tau}(e^{-z\tau}) \rho _{1}\left( e^{-z\tau} \right) \tau ^{2} \right]  dz.
\end{split}
\end{equation*}
According to the triangle inequality and Lemma \ref{SDELemma2.7}, it yields
\begin{equation*}
\begin{split}
\left\| J_{1}  \right\| \le& c\int_{\frac{\pi }{\tau \sin \theta } }^{+\infty }e^{rt_{n}\cos \theta }r^{-\left( \alpha +\gamma  \right) }dr \leq c\tau ^{2} \int_{\frac{\pi }{\tau \sin \theta } }^{+\infty }e^{rt_{n}\cos \theta }r^{2-\left( \alpha +\gamma  \right) }dr \\
\le& c\tau ^{2}t_{n}^{\alpha +\gamma -3}  \int_{\frac{t_{n} \pi }{\tau \sin \theta } }^{+\infty }e^{s\cos \theta }s^{2-\left( \alpha +\gamma  \right) }ds \leq c\tau ^{2}t_{n}^{\alpha +\gamma -3}
\end{split}
\end{equation*}
and
\begin{equation*}
\begin{split}
\left\| J_{2}  \right\|
\leq & c\int_{k}^{\frac{\pi }{\tau \sin \theta } }e^{rt_{n}\cos \theta  } \tau ^{2}r^{2-\gamma -\alpha } dr + c \int_{-\theta }^{\theta }e^{kt_{n}\cos \psi  }  \tau ^{2}k^{3-\gamma -\alpha} d\psi
\leq c\tau ^{2}t_{n}^{\alpha +\gamma -3}.
\end{split}
\end{equation*}
Then, we have
\begin{equation}\label{SDE2.13}
\left\| \left( \left( \mathscr{E}-\mathscr{E}_{\tau } \right)\ast t \right) \left( t_{n}  \right)  \right\|
\leq  c\tau ^{2}t_{n}^{\alpha +\gamma -3}.
\end{equation}
Similarly, we obtain
\begin{equation}\label{SDE2.14}
\left\| \left( t-t_{n} \right) \left(\left( \mathscr{E}-\mathscr{E}_{\tau } \right)\ast 1\right) ( t_{n} ) \right\|  \leq  c\tau ^{2}t_{n}^{\alpha +\gamma -3}.
\end{equation}

According to~\eqref{SDE1.8} and~\eqref{SDE1.4}, it leads to
\begin{equation*}
\left\| \mathscr{E}(t) \right\| \leq  c \int_{\Gamma_{\theta,\kappa}} \left| e^{zt} \right| |z|^{2-\alpha-\gamma} |dz| \leq c t^{\alpha+\gamma-3}.
\end{equation*}
Moreover, one has
\begin{equation}\label{SDE2.15}
\begin{split}
\left\| \int _{t}^{t_{n} } \left( s-t \right)  \mathscr{E}\left( s \right)  ds \right\|
\leq c \int_{t}^{t_{n}}\left( s-t \right)s^{\alpha +\gamma -3}ds.
\end{split}
\end{equation}
By the definition of $\mathscr{E}_{\tau } \left( t \right) = \sum_{n=0}^{\infty } \mathscr{E}_{\tau }^{n} \delta_{t_{n} }\left( t \right)$ and~\eqref{SDE2.11}, we deduce
\begin{equation}\label{SDE2.16}
\begin{split}
\left\| \int _{t}^{t_{n} } \left( s-t \right) \mathscr{E}_{\tau}\left( s \right) ds \right\|
\leq (t_{n}-t) \left\| \mathscr{E}^{n}_{\tau} \right\|
\leq c\tau^{2} t_{n}^{\alpha+\gamma-3}.
\end{split}
\end{equation}

Case I:\@ $\alpha+\gamma>5/2$. From~\eqref{SDE2.13}-\eqref{SDE2.16}, we get
\begin{equation*}
\left\| \left( \left( \mathscr{E}-\mathscr{E}_{\tau } \right)\ast t \right) \left( t  \right)  \right\|
\leq  c\tau ^{2}t^{\alpha +\gamma -3}
\end{equation*}
and
\begin{equation*}
\int_{0}^{t_{n} }\left\| \left( \left( \mathscr{E}-\mathscr{E}_{\tau } \right)\ast t  \right) (s) \right\|^{2}ds
\leq c \tau^{4} \int_{0}^{t_{n} }  s^{2\alpha +2\gamma -6}  ds
\leq c\tau ^{4}  t_{n}^{2\alpha +2\gamma -5}.
\end{equation*}

Case II:\@ $1 < \alpha+\gamma \leq 5/2$. For small $\varepsilon$, 
according to~\eqref{SDE2.13}-\eqref{SDE2.16}, it yields
\begin{equation*}
\left\| \left( \left( \mathscr{E}-\mathscr{E}_{\tau} \right)\ast t \right) \left( t  \right)  \right\|
\leq  c \tau^{\alpha+\gamma-\frac{1}{2}-\varepsilon}t^{-\frac{1}{2}+\varepsilon}
\end{equation*}
and
\begin{equation*}
\int_{0}^{t_{n} }\left\| \left( \left( \mathscr{E}-\mathscr{E}_{\tau } \right)\ast t  \right) (s)  \right\|^{2}ds
\leq c\tau^{2\alpha+2\gamma-1-2\varepsilon} \int_{0}^{t_{n} }  s^{-1+2\varepsilon}  ds
\leq c\tau^{2\alpha+2\gamma-1-2\varepsilon}  t_{n}^{2\varepsilon}.
\end{equation*}
The proof is completed.
\end{proof}

\begin{theorem}\label{SDETheorem2.9}
Let $V_{l}(t_{n})$ and $V_{l}^{n}$ be the solutions of~\eqref{SDE2.5} and~\eqref{SDE2.7}, respectively.
Let $\upsilon, b \in L^{2}(\Omega)$. 
Then for any small $\varepsilon>0$  and  $1< \alpha+\gamma \leq \frac{5}{2}$,
\begin{equation*}
\mathbb{E} \left\|V_{l}(t_{n})-V_{l}^{n}\right\|^{2}
\leq c\tau^{4} t_{n}^{-4} \left\| \upsilon \right\|_{L^2(\Omega)} + c\tau^{4} t_{n}^{-2} \left\| b \right\|_{L^2(\Omega)} + c \sum_{j=1}^{\infty}\left( \mu_{j}^{l} \right)^{2} \tau^{2\alpha+2\gamma-1-2\varepsilon}  t_{n}^{2\varepsilon},
\end{equation*}
and for $\frac{5}{2} < \alpha+\gamma <  3$,
\begin{equation*}
\mathbb{E} \left\|V_{l}(t_{n})-V_{l}^{n}\right\|^{2}
\leq c\tau^{4} t_{n}^{-4} \left\| \upsilon \right\|_{L^2(\Omega)} + c\tau^{4} t_{n}^{-2} \left\| b \right\|_{L^2(\Omega)} + c \sum_{j=1}^{\infty }\left( \mu_{j}^{l} \right)^{2} \tau^{4} t_{n}^{2\alpha +2\gamma-5}.
\end{equation*}
\end{theorem}
\begin{proof}
Subtracting~\eqref{SDE2.5} from~\eqref{SDE2.7}, we have the following split
\begin{equation*}
V_{l}^{n}-V_{l}(t_{n}) = I_{1} - I_{2} + I_{3} - I_{4} + I_{5}.
\end{equation*}
Here the related initial terms are given in
\begin{equation*}
\begin{split}
I_{1} & = \frac{1}{2\pi i} \int_{\Gamma^{\tau}_{\theta,\kappa}} e^{zt_{n}} \left(\delta^{\alpha}_{\tau}(e^{-z\tau}) -A\right)^{-1} \delta_{\tau}(e^{-z\tau}) \rho_{1}( e^{-z\tau}) \tau^{2} A - (z^{\alpha}-A)^{-1} z^{-1} A \upsilon dz,\\
I_{2} & = \frac{1}{2\pi i} \int_{\Gamma_{\theta,\kappa}\backslash \Gamma^{\tau}_{\theta,\kappa}} e^{zt_{n}} ( z^{\alpha}-A)^{-1} z^{-1} A\upsilon dz,
\end{split}
\end{equation*}
\begin{equation*}
\begin{split}
I_{3} & = \frac{1}{2\pi i} \int_{\Gamma^{\tau}_{\theta,\kappa}} e^{zt_{n}} \left(\delta^{\alpha}_{\tau}(e^{-z\tau}) -A\right)^{-1} \delta_{\tau}(e^{-z\tau}) \frac{\rho_{2}(e^{-z\tau})}{2} \tau^{3} A - (z^{\alpha}-A)^{-1} z^{-2} A  b dz, \quad\\
I_{4} & = \frac{1}{2\pi i} \int_{\Gamma_{\theta,\kappa}\backslash \Gamma^{\tau}_{\theta,\kappa}} e^{zt_{n}} ( z^{\alpha}-A)^{-1} z^{-2}Ab dz.
\end{split}
\end{equation*}
The related noise term is
\begin{equation*}
\begin{split}
I_{5}
& = \frac{\tau}{2\pi i}\int_{\Gamma^{\tau}_{\theta,\kappa}} e^{zt_n} \left( \delta^{\alpha}_{\tau}(e^{-z\tau}) -A\right)^{-1} \delta^{2-\gamma}_{\tau}(e^{-z\tau}) \widetilde{g_{l}} (e^{-z\tau}) dz \\
& \quad - \frac{1}{2\pi i} \int_{\Gamma_{\theta, \kappa}} e^{zt_n} ( z^{\gamma}-A)^{-1} z^{2-\gamma} \widehat{g_{l}}(z) dz,
\end{split}
\end{equation*}
which is estimated in Theorem \ref{SDETheorem2.8}.

From Lemma \ref{SDELemma2.7}, we estimate $I_{1}$ and $I_{3}$ as following
\begin{equation*}
\left\|I_{1} \right\|_{L^2(\Omega)} \leq c\tau^{2} t_{n}^{-2} \left\| \upsilon \right\|_{L^2(\Omega)} \quad {\rm and} \quad \left\| I_{3} \right\|_{L^2(\Omega)} \leq c\tau^{2} t_{n}^{-1} \left\| b \right\|_{L^2(\Omega)}.
\end{equation*}
Using the resolvent estimate~\eqref{SDE1.4}, we estimate $I_{2}$ and $I_{4}$ as following
\begin{equation*}
\left\|I_{2} \right\|_{L^2(\Omega)} \leq   c\tau^{2} t_{n}^{-2} \left\| \upsilon \right\|_{L^2(\Omega)} \quad {\rm and} \quad \left\|I_{4} \right\|_{L^2(\Omega)} \leq   c\tau^{2} t_{n}^{-1} \left\| b \right\|_{L^2(\Omega)}.
\end{equation*}
Thus, for $1< \alpha+\gamma \leq \frac{5}{2}$, we have
\begin{equation*}
\mathbb{E} \left\| V_{l}(t_{n})-V_{l}^{n} \right\|^{2}
\leq c\tau^{4} t_{n}^{-4} \left\| \upsilon \right\|_{L^2(\Omega)} + c\tau^{4} t_{n}^{-2} \left\| b \right\|_{L^2(\Omega)} + c \sum_{j=1}^{\infty}\left( \mu_{j}^{l} \right)^{2} \tau^{2\alpha + 2\gamma-1-2\varepsilon}  t_{n}^{2\varepsilon},
\end{equation*}
and for $\frac{5}{2} < \alpha+\gamma <  3$, it yields
\begin{equation*}
\mathbb{E} \left\|V_{l}(t_{n})-V_{l}^{n}\right\|^{2}
\leq c\tau^{4} t_{n}^{-4} \left\| \upsilon \right\|_{L^2(\Omega)} + c\tau^{4} t_{n}^{-2} \left\| b \right\|_{L^2(\Omega)} + c \sum_{j=1}^{\infty }\left( \mu_{j}^{l} \right)^{2} \tau^{4} t_{n}^{2\alpha +2\gamma -5}.
\end{equation*}
The proof is completed.
\end{proof}

\begin{theorem}\label{SDETheorem2.10}
Let $V (t_{n})$ and $V_{l}^{n}$ be the solutions of~\eqref{SDE2.2} and~\eqref{SDE2.7}, respectively.
Let $\upsilon, b \in L^{2}(\Omega)$. 
Then for any small $\varepsilon>0$, it holds
\begin{equation*}
\begin{split}
\mathbb{E} \left\|V(t_{n})-V_{l}^{n}\right\|^{2}
&\leq  c\tau^{4} t_{n}^{-4} \left\| \upsilon \right\|_{L^2(\Omega)} + c \tau^{4} t_{n}^{-2} \left\| b \right\|_{L^2(\Omega)}\\
&\quad +c \sum_{j=1}^{\infty} {(\eta_{j}^{l})}^{2} +  c \sum_{j=1}^{\infty}\left( \mu_{j}^{l} \right)^{2}  \tau^{2\alpha+2\gamma-1-2\varepsilon}  t_{n}^{2\varepsilon},~~1< \alpha+\gamma \leq \frac{5}{2};
\end{split}
\end{equation*}
and
\begin{equation*}
\begin{split}
\mathbb{E} \left\|V(t_{n})-V_{l}^{n}\right\|^{2}
&\leq  c\tau^{4} t_{n}^{-4} \left\| \upsilon \right\|_{L^2(\Omega)} + c \tau^{4} t_{n}^{-2} \left\| b \right\|_{L^2(\Omega)}\\
&\quad +c \sum_{j=1}^{\infty} {(\eta_{j}^{l})}^{2} + c \sum_{j=1}^{\infty }\left( \mu_{j}^{l} \right)^{2} \tau^{4} t_{n}^{2\alpha +2\gamma-5},~~\frac{5}{2} < \alpha+\gamma <  3.
\end{split}
\end{equation*}
\end{theorem}
\begin{proof}
By Theorems \ref{SDETheorem2.2}, \ref{SDETheorem2.9} and the triangle inequality, the proof is completed.
\end{proof}

\begin{remark}\label{SDERemark2.11}
For ID$1$-BDF$2$ method,
\begin{equation}\label{SDE2.17}
{\partial}_{\tau}^{\alpha} V_{l}^{n} - A V_{l}^{n} = {\partial}_{\tau} \left( t_{n} A\upsilon + \frac{t_{n}^{2}}{2} Ab \right) + {\partial}_{\tau}^{1- \gamma} g_{l} (t_{n}), \quad \alpha \in (1,2)
\end{equation}
with $g_{l}(t) = 1 \ast \frac{d W_{l}(t)}{ dt} $, the similar proof with order $O(\tau^{\min\{\alpha + \gamma -1/2,1\}})$  can be obtained by Theorem \ref{SDETheorem2.10}.
\end{remark}

\begin{remark}\label{SDERemark2.12}
From Theorem \ref{SDETheorem2.10}, for the subdiffusion case $\alpha \in (0, 1)$, the proof with order  $O(\tau^{\alpha + \gamma -1/2})$ can be derived by ID$2$-BDF$2$ method for any $\alpha + \gamma > 1/2$.
\end{remark}

\section{ID$3$-BDF$3$ method}\label{Sec3}
To break the second-order barrier in Section \ref{Sec2}, we further propose the ID$3$-BDF3 method.
Note that BDF$3$ methods are $A(\vartheta)$-stable with $\vartheta\approx 86.03^\circ$; see \cite{Akrivisenergy2021,HairerSolving2010}.
Thus, for the fractional diffusion-wave equation, the BDF3 scheme is unconditionally stable for any $\alpha < \alpha^{\ast}:=\pi/(\pi-\vartheta) = 1.91$.
For $\alpha \geq \alpha^{\ast}$, it is conditionally stable, see Condition 3.1 in \cite[p. A3137]{JinCorrection2017}, which also require to be fulfilled in this section.

\subsection{Numerical scheme and solution representation}\label{Sec3.1}
With $g(t): = \frac{t^2}{2} \ast \frac{d W(t)}{dt}$, we rewrite \eqref{SDE1.7} as
\begin{equation}\label{SDE3.1}
\partial^{\alpha}_t V(t) - A V (t) = \partial^{2}_{t} \left( \frac{t^{2}}{2} A\upsilon + \frac{t^{3}}{6} Ab \right) + \partial^{ 3-\gamma}_{t} g(t), \quad t \in (0, T].
\end{equation}

With $g_{l}(t)$ defined by \eqref{g_l_1}, we rewrite \eqref{SDE2.4} as
\begin{equation}\label{SDE3.2}
\partial^{\alpha}_{t} V_{l}(t) - A V_{l}(t) = \partial^{2}_{t} \left( \frac{t^{2}}{2} A\upsilon + \frac{t^{3}}{6} Ab \right) + \partial^{ 3-\gamma}_{t} g_{l} (t), \quad t \in (0, T].
\end{equation}
Then ID$3$-BDF$3$ method for~\eqref{SDE3.2} is designed by
\begin{equation}\label{SDE3.3}
\partial^{\alpha}_{\tau} V_{l}^{n} - AV_{l}^{n}= \partial^{2}_{\tau} \left( \frac{t_{n}^{2}}{2} A\upsilon + \frac{t_{n}^{3}}{6} Ab \right) + \partial^{3-\gamma}_{\tau}  g_{l} (t_{n}).
\end{equation}
where ${\partial}_{\tau}^{\beta} \varphi^{n}$, $\beta= \alpha, 2, 3- \gamma $ are defined by \eqref{SDE1.13}.

Similar to the proof of Lemma \ref{SDELemma2.4}, for the discrete solution representation in \eqref{SDE3.3}, we have the following results.
\begin{lemma}\label{SDELemma3.1}
Let $\delta_{\tau}(\xi)$ be given in~\eqref{SDE1.14} and $\rho_{j}(\xi)=\sum^{\infty}_{n=1}n^{j} \xi^{n}$ with $j=2, 3$.
Let $g_{l}(t)$ be defined by \eqref{g_l_1}. 
Then the discrete solution of~\eqref{SDE3.3} is represented by
\begin{align}
V_{l}^{n}
& =\frac{1}{2\pi i}\int_{\Gamma^{\tau}_{\theta,\kappa}} e^{zt_n} {\left(\delta^{\alpha}_{\tau}(e^{-z\tau})-A\right)}^{-1} \delta^{2}_{\tau}(e^{-z\tau}) \tau \left( \frac{\rho_{2}(e^{-z\tau})}{2} \tau^{2} A \upsilon +  \frac{\rho_{3}(\xi)}{6} \tau^{3} Ab  \right) dz  \notag \\
& \quad + \frac{1}{2\pi i}\int_{\Gamma^{\tau}_{\theta,\kappa}} e^{zt_n} {\left(\delta^{\alpha}_{\tau}(e^{-z\tau})-A\right)}^{-1} \delta^{3-\gamma}_{\tau}(e^{-z\tau}) \tau  \widetilde{g_{l}}(e^{-z\tau}) dz \label{SDE3.4}
\end{align}
with $\Gamma^{\tau}_{\theta, \kappa}=\{z\in \Gamma_{\theta, \kappa}: |\Im z|\leq \pi / \tau\}$. 
\end{lemma}

\subsection{Error estimates}\label{Sec3.2}
We give some lemmas which will be used in the error estimates.
\begin{lemma}\cite{JinCorrection2017}\label{SDELemma3.2}
Let $\delta_{\tau}(\xi)$ be given in~\eqref{SDE1.14}. Then there exist the positive constants $c_{1},c_{2}$, $c$ and $\vartheta=86.03^{\circ}$, $\theta \in (\pi/2, \theta_{\varepsilon})$ with $\theta_{\varepsilon} \in (\pi/2, \pi),~\forall \varepsilon>0$ such that
\begin{equation*}
\begin{split}
& c_{1}|z|\leq \left| \delta_{\tau}(e^{-z\tau}) \right| \leq c_{2}|z|, \quad \left| \delta_{\tau}(e^{-z\tau})-z \right| \leq c \tau^{3}|z|^{4},\\
& \left| \delta^{\alpha}_{\tau}(e^{-z\tau})-z^{\alpha} \right| \leq c \tau^{3}|z|^{3+\alpha}, \quad \delta_{\tau}(e^{-z\tau}) \in \Sigma_{\pi - \vartheta +\varepsilon} \quad {\forall} z\in \Gamma^{\tau}_{\theta,\kappa}.
\end{split}
\end{equation*}
\end{lemma}

\begin{lemma}\label{SDELemma3.3}
Let $\delta_{\tau}(\xi)$ be given in~\eqref{SDE1.14} and $\rho_{2}(\xi)=\sum^{\infty}_{n=1} n^{2} \xi^{n}$. Then there exists a positive constants $c$ such that
\begin{equation*}
\left|  \frac{\rho_{2}(e^{-z\tau})}{2} \tau^{3} - z^{-3} \right| \leq c \tau^{3}~
~{\rm and}~~\left| \delta^{3-\gamma}_{\tau}(e^{-z\tau}) \frac{\rho_{2}(e^{-z\tau})}{2} \tau^{3} - z^{-\gamma} \right| \leq c \tau^{3}|z|^{3-\gamma}~~  \forall z\in \Gamma^{\tau}_{\theta,\kappa}
\end{equation*}
where $\theta\in (\pi/2,\pi)$ is sufficiently close to $\pi/2$.
\end{lemma}
\begin{proof}
The similar arguments can be performed as Lemma \ref{SDELemma2.6}, we
omit it here.
\end{proof}

\begin{lemma}\label{SDELemma3.4}
Let $\delta_{\tau}(\xi)$ be given in~\eqref{SDE1.14} and $\rho_{j}(\xi)=\sum^{\infty}_{n=1}n^{j} \xi^{n}$, $j=2, 3$.
Then there exists a positive constant $c$ such that
\begin{equation*}
\left\| {\left( \delta^{\alpha}_{\tau}(e^{-z\tau}) - A\right)}^{-1} \delta^{2}_{\tau}(e^{-z\tau}) \frac{\rho_{j}(e^{-z\tau})}{\Gamma(j+1)} \tau^{j+1} A - {(z^{\alpha}-A)}^{-1} z^{1-j} A \right\| \leq  c \tau^{3}|z|^{4-j} 
\end{equation*}
and
\begin{equation*}
\left\| {\left( \delta^{\alpha}_{\tau}(e^{-z\tau}) - A\right)}^{-1} \delta^{3-\gamma}_{\tau}(e^{-z\tau}) \frac{\rho_{2}(e^{-z\tau})}{2} \tau^{3} - {(z^{\alpha}-A)}^{-1} z^{-\gamma} \right\|
 \leq  c \tau^{3}|z|^{3-\alpha-\gamma}.
\end{equation*}
\end{lemma}
\begin{proof}
Form Lemma \ref{SDELemma3.3}, the similar arguments can be performed as Lemma \ref{SDELemma2.7}, we
omit it here.
\end{proof}

\begin{theorem}\label{SDETheorem3.5}
Let $V_{l}(t_{n})$ and $V_{l}^{n}$ be the solutions of~\eqref{SDE3.2} and~\eqref{SDE3.4}, respectively.
Let $\upsilon = b = 0$.
Then for any small $\varepsilon>0$
\begin{equation*}
\mathbb{E} \left\|V_{l}(t_{n})-V_{l}^{n}\right\|^{2} \leq c \sum_{j=1}^{\infty}\left( \mu_{j}^{l} \right)^{2} \tau^{2\alpha+2\gamma-1-2\varepsilon}  t_{n}^{2\varepsilon}.
\end{equation*}
\end{theorem}
\begin{proof}
By~\eqref{SDE3.2}, we obtain
\begin{equation}\label{SDE3.5}
\begin{split}
V_{l}(t_{n}) = \left(\mathscr{E}(t)\ast \left( \frac{t^{2}}{2} \ast \frac{d W_{l}(t)}{dt} \right)\right)(t_{n})
=& \left(\left(\mathscr{E}(t)\ast \frac{t^{2}}{2} \right) \ast \frac{d W_{l}(t)}{dt} \right)(t_{n})
\end{split}
\end{equation}
with
\begin{equation}\label{SDE3.6}
\mathscr{E}(t)= \frac{1}{2\pi i} \int_{\Gamma_{\theta,\kappa}} e^{zt}(z^{\alpha} - A)^{-1} z^{3-\gamma} dz.
\end{equation}
Using $g_{l}^{n} = g_{l}(t_{n})$ with $g_{l}(t)$ defined by \eqref{g_l_1} and similar to the proof of Lemma \ref{SDELemma2.4}, it yields
\begin{equation*}
\begin{split}
\widetilde{V_{l}}(\xi)
&=\left( \delta^{\alpha}_{\tau}(\xi) -A\right)^{-1} \delta^{3-\gamma}_{\tau}(\xi) \widetilde{g_{l}}(\xi) = \widetilde{\mathscr{E_{\tau}}}(\xi)\widetilde{g_{l}}(\xi)
=\sum^{\infty}_{n=0}\mathscr{E}^{n}_{\tau}\xi^{n}\sum^{\infty}_{j=0} g_{l}^{j} \xi^{j}\\
&=\sum^{\infty}_{n=0}\sum^{\infty}_{j=0}\mathscr{E}^{n}_{\tau} g_{l}^{j} \xi^{n+j}=\sum^{\infty}_{j=0}\sum^{\infty}_{n=j}\mathscr{E}^{n-j}_{\tau} g_{l}^{j} \xi^{n}
=\sum^{\infty}_{n=0}\sum^{n}_{j=0}\mathscr{E}^{n-j}_{\tau} g_{l}^{j} \xi^{n}=\sum^{\infty}_{n=0}V_{l}^n\xi^{n}
\end{split}
\end{equation*}
with
\begin{equation*}
V_{l}^{n}=\sum^{n}_{j=0}\mathscr{E}^{n-j}_{\tau} g_{l}^{j} :=\sum^{n}_{j=0}\mathscr{E}^{n-j}_{\tau} g_{l}(t_{j}).
\end{equation*}
Here $\sum^{\infty}_{n=0}\mathscr{E}^{n}_{\tau}\xi^{n}=\widetilde{\mathscr{E_{\tau}}}(\xi)=\left( \delta^{\alpha}_{\tau}(\xi) -A\right)^{-1} \delta^{3-\gamma}_{\tau}(\xi)$.
By the Cauchy's integral formula and the change of variables $\xi=e^{-z\tau}$, we obtain the representation of the $\mathscr{E}^{n}_{\tau}$ as 
 \begin{equation*}
\mathscr{E}^{n}_{\tau}=\frac{1}{2\pi i}\int_{|\xi|=\rho}{\xi^{-n-1}\widetilde{\mathscr{E_{\tau}}}(\xi)}d\xi
=\frac{\tau}{2\pi i}\int_{\Gamma^{\tau}_{\theta,\kappa}} {e^{zt_n}\left( \delta^{\alpha}_{\tau}(e^{-z\tau}) -A\right)^{-1} \delta^{3-\gamma}_{\tau}(e^{-z\tau}) }dz,
\end{equation*}
where $\theta\in (\pi/2,\pi)$ is sufficiently close to $\pi/2$ and $\kappa=t_{n}^{-1}$ in~\eqref{SDE1.9}.

According to Lemma \ref{SDELemma3.2}, \eqref{SDE2.9} and $\tau t^{-1}_{n} = \frac{1}{n}\leq 1$, there exists
\begin{equation}\label{SDE3.7}
\|\mathscr{E}^{n}_{\tau}\| \leq c \tau \left[\int^{\frac{\pi}{\tau\sin\theta}}_{\kappa} e^{rt_{n}\cos\theta} r^{3-\alpha-\gamma}dr +\int^{\theta}_{-\theta}e^{\kappa t_{n}\cos\psi} \kappa^{4-\alpha-\gamma}  d\psi \right]
\leq c\tau t_{n}^{\alpha+\gamma-4}.
\end{equation}
Let $ \mathscr{E}_{\tau}(t)=\sum^{\infty}_{n=0}\mathscr{E}^{n}_{\tau}\delta_{t_{n}}(t)$, with $\delta_{t_{n}}$ being the Dirac delta function at $t_{n}$.
Then
\begin{equation}\label{SDE3.8}
(\mathscr{E}_{\tau}(t)\ast g_{l}(t))(t_{n})
 = \left(\sum^{\infty}_{j=0}\mathscr{E}^{j}_{\tau}\delta_{t_{j}}(t) \ast g_{l}(t) \right)(t_{n})
 = \sum^{n}_{j=0}\mathscr{E}^{n-j}_{\tau} g_{l}(t_{j})=V_{l}^{n}.
\end{equation}
Moreover, using the above equation, there exists
\begin{equation*}
\begin{split}
\widetilde{(\mathscr{E}_{\tau}\ast t^{2})}(\xi)
& = \sum^{\infty}_{n=0} \sum^{n}_{j=0}\mathscr{E}^{n-j}_{\tau}t^{2}_{j}\xi^{n}  =\sum^{\infty}_{j=0} \sum^{\infty}_{n=j}\mathscr{E}^{n-j}_{\tau}t^{2}_{j}\xi^{n}
  =\sum^{\infty}_{j=0} \sum^{\infty}_{n=0}\mathscr{E}^{n}_{\tau}t^{2}_{j}\xi^{n+j}\\
& =\sum^{\infty}_{n=0}\mathscr{E}^{n}_{\tau}\xi^{n}\sum^{\infty}_{j=0}t^{2}_{j}\xi^{j}  =\widetilde{\mathscr{E_{\tau}}}(\xi) \tau^{2} \sum^{\infty}_{j=0}j^{2} \xi^{j}
  =\widetilde{\mathscr{E}_{\tau}}(\xi) \tau^{2} \rho_{2}(\xi).
\end{split}
\end{equation*}
According to \eqref{SDE1.2}, \eqref{SDE3.5}, \eqref{SDE3.8} and It\^{o} isometry property, we have
\begin{equation*}\label{eq3.25}
\begin{split}
\mathbb{E} \left\|V_{l}(t_{n})-V_{l}^{n}\right\|^{2}
= & \mathbb{E} \left\| \left( \left[ \left( \mathscr{E}-\mathscr{E}_{\tau} \right) \ast \frac{t^{2}}{2} \right] \ast
\frac{d W_{l}(t)}{dt}  \right) \left( t_{n}  \right) \right\|^{2} \\
= & \mathbb{E} \left\| \sum_{j=1}^{\infty}\int_{0}^{t_{n}} \left[ \left( \left( \mathscr{E}-\mathscr{E}_{\tau} \right)\ast \frac{t^{2}}{2} \right) \left( t_{n}-s  \right) \right] \sigma_{j}^{l}\left( s \right)  \varphi_{j} d \beta_{j}\left( s \right) \right\|^{2} \\
= & \sum_{j=1}^{\infty} \int_{0}^{t_{n} } \left\| \left[ \left( \left( \mathscr{E}-\mathscr{E}_{\tau}  \right)\ast \frac{t^{2}}{2} \right) \left( t_{n}-s  \right) \right] \sigma_{j}^{l}\left( s \right) \varphi_{j} \right\|^{2} ds \\
\leq & \sum_{j=1}^{\infty} \left( \mu_{j}^{l} \right)^{2} \int_{0}^{t_{n} }\left\| \left( \left( \mathscr{E}-\mathscr{E}_{\tau} \right) \ast \frac{t^{2}}{2}  \right) \left( t_{n}-s  \right)  \right\|^{2}ds.
\end{split}
\end{equation*}

Next, we prove the following inequality for any $t\in \left( t_{n-1},t_{n} \right)$ 
\begin{equation*}
\left\| \left( \left( \mathscr{E}-\mathscr{E}_{\tau} \right)\ast \frac{t^{2}}{2} \right) \left( t  \right)  \right\| \leq  c \tau^{\alpha+\gamma-\frac{1}{2}-\varepsilon}t^{-\frac{1}{2}+\varepsilon}.
\end{equation*}
By Taylor series expansion of $\mathscr{E} \left( t \right)$ at $t=t_{n} $, we get
\begin{equation*}
\begin{split}
\left(\mathscr{E} \left( t \right)\ast \frac{t^{2}}{2} \right)(t)
=& \left(\mathscr{E} \left( t \right)\ast \frac{t^{2}}{2} \right) ( t_{n} ) + \left( t-t_{n} \right) \left(\mathscr{E} \left( t \right)\ast t\right) (t_{n}) \\
& + \frac{\left( t-t_{n} \right)^{2}}{2} \left(\mathscr{E} \left( t \right)\ast 1\right) ( t_{n} ) + \int _{t}^{t_{n} } \frac{\left( s-t \right)^{2}}{2} \mathscr{E}\left( s \right)  ds,
\end{split}
\end{equation*}
which also holds for the convolution $\left( \mathscr{E}_{\tau}\ast \frac{t^{2}}{2} \right)\left( t \right)$.
Now, we estimate
\begin{equation*}
\begin{split}
\left\| \left( \left( \mathscr{E}-\mathscr{E}_{\tau } \right)\ast \frac{t^{2}}{2} \right) \left( t_{n}  \right)  \right\| \leq \left\| J_{1}  \right\| +\left\| J_{2} \right\|
\end{split}
\end{equation*}
with
\begin{equation*}
\begin{split}
J_{1}=&\frac{1}{2\pi i}\int_{\Gamma_{\theta ,k}\setminus \Gamma_{\theta ,k}^{\tau}} e^{zt_{n}}\left( z^{\alpha } -A \right)^{-1}\frac{1}{z^{\gamma} } dz, \\
J_{2}=&\frac{1}{2\pi i}\int_{\Gamma_{\theta ,k}^{\tau}} e^{zt_{n}}\left[  \left( z^{\alpha } -A \right)^{-1} z^{-\gamma}  -\left( \delta^{\alpha}_{\tau}(e^{-z\tau}) -A\right)^{-1} \delta^{3-\gamma}_{\tau}(e^{-z\tau}) \frac{\rho_{2}\left( e^{-z\tau} \right)}{2} \tau ^{3} \right]  dz.
\end{split}
\end{equation*}
According to the triangle inequality and Lemma \ref{SDELemma3.4}, it yields 
\begin{equation*}
\left\| J_{1}  \right\|
\leq c \int_{\frac{\pi }{\tau \sin \theta} }^{+\infty} e^{rt_{n} \cos \theta}r^{-\left( \alpha + \gamma \right) }dr
\leq c \tau^{3}t_{n}^{\alpha +\gamma -4}  \int_{\frac{t_{n} \pi }{\tau \sin \theta } }^{+\infty}e^{s\cos \theta }s^{3-\left( \alpha + \gamma \right) }ds \leq c\tau^{3}t_{n}^{\alpha +\gamma -4},
\end{equation*}
and
\begin{equation*}
\begin{split}
\left\| J_{2} \right\|
\leq & c \int_{k}^{\frac{\pi}{\tau \sin \theta} }e^{rt_{n} \cos \theta} \tau^{3}r^{3-\gamma -\alpha} dr + c \int_{-\theta }^{\theta }e^{kt_{n}\cos \psi } \tau ^{3}k^{4-\gamma -\alpha} d\psi
\leq c \tau^{3}t_{n}^{\alpha +\gamma -4}.
\end{split}
\end{equation*}
Then, we have
\begin{equation}\label{SDE3.9}
\left\| \left( \left( \mathscr{E}-\mathscr{E}_{\tau} \right) \ast \frac{t^{2}}{2} \right) \left( t_{n}  \right)  \right\| \leq  c \tau^{3}t_{n}^{\alpha +\gamma -4}.
\end{equation}
Similarly, one has
\begin{equation}\label{SDE3.10}
\begin{split}
& \left\| \left(t-t_{n}\right) \left(\left(\mathscr{E}-\mathscr{E}_{\tau} \right)\ast t\right)(t_{n}) \right\| \leq c \tau^{3}t_{n}^{\alpha +\gamma -4}, \\
& \left\| \left(t-t_{n}\right)^{2} \left(\left(\mathscr{E}-\mathscr{E}_{\tau} \right)\ast 1\right)(t_{n})\right\| \leq c \tau^{3}t_{n}^{\alpha +\gamma -4}.
\end{split}
\end{equation}

According to~\eqref{SDE3.6} and~\eqref{SDE1.4}, it leads to
\begin{equation*}
\left\| \mathscr{E}(t) \right\| \leq c \int_{\Gamma_{\theta,\kappa}} \left| e^{zt} \right| |z|^{3-\alpha-\gamma} |dz| \leq c t^{\alpha+\gamma-4}.
\end{equation*}
Moreover, we obtain
\begin{equation}\label{SDE3.11}
\left\| \int_{t}^{t_{n} } \frac{\left( s-t \right)^{2}}{2}  \mathscr{E}\left( s \right)  ds \right\|
\leq c\int_{t}^{t_{n}} \frac{\left( s-t \right)^{2}}{2} s^{\alpha +\gamma -4}ds.
\end{equation}
By the definition of $\mathscr{E}_{\tau } \left( t \right) =\sum_{n=0}^{\infty } \mathscr{E}_{\tau }^{n} \delta _{t_{n} }\left( t \right)$ and~\eqref{SDE3.7}, we deduce
\begin{equation}\label{SDE3.12}
\left\| \int _{t}^{t_{n} } \frac{\left( s-t \right)^{2}}{2} \mathscr{E}_{\tau}\left( s \right) ds \right\|
\leq \frac{(t_{n}-t)^{2}}{2} \left\| \mathscr{E}^{n}_{\tau} \right\|
\leq c\tau^{3} t_{n}^{\alpha+\gamma-4}.
\end{equation}
For small $\varepsilon$,
according to~\eqref{SDE3.9}-\eqref{SDE3.12}, it yields
\begin{equation*}
\left\| \left( \left( \mathscr{E}-\mathscr{E}_{\tau} \right)\ast \frac{t^{2}}{2} \right) \left( t \right) \right\|
\leq  c \tau^{\alpha+\gamma-\frac{1}{2}-\varepsilon}t^{-\frac{1}{2}+\varepsilon}
\end{equation*}
and
\begin{equation*}
\int_{0}^{t_{n} }\left\| \left( \left( \mathscr{E}-\mathscr{E}_{\tau } \right)\ast \frac{t^{2}}{2} \right) (s) \right\|^{2}ds
\leq c\tau^{2\alpha+2\gamma-1-2\varepsilon} \int_{0}^{t_{n} }  s^{-1+2\varepsilon}  ds
\leq c\tau^{2\alpha+2\gamma-1-2\varepsilon} t_{n}^{2\varepsilon}.
\end{equation*}
The proof is completed.
\end{proof}

\begin{theorem}\label{SDETheorem3.6}
Let $V_{l}(t_{n})$ and $ V_{l}^{n}$ be the solutions of~\eqref{SDE3.2} and~\eqref{SDE3.4}, respectively.
Let $\upsilon, b \in L^{2}(\Omega)$.
Then for any small $\varepsilon>0$
\begin{equation*}
\mathbb{E} \left\|V_{l}(t_{n})-V_{l}^{n}\right\|^{2}
\leq c\tau^{6} t_{n}^{-6} \left\| \upsilon \right\|_{L^2(\Omega)} + c \tau^{6} t_{n}^{-4} \left\| b \right\|_{L^2(\Omega)} + c \sum_{j=1}^{\infty}\left( \mu_{j}^{l} \right)^{2} \tau^{2\alpha+2\gamma-1-2\varepsilon}  t_{n}^{2\varepsilon}.
\end{equation*}
\end{theorem}
\begin{proof}
Subtracting~\eqref{SDE3.2} from~\eqref{SDE3.4}, we have the following split
\begin{equation*}
V_{l}^{n}- V_{l}(t_{n}) = I_{1} - I_{2} + I_{3} - I_{4} + I_{5}.
\end{equation*}
Here the related initial terms are
\begin{equation*}
\begin{split}
I_{1} & = \frac{1}{2\pi i} \int_{\Gamma^{\tau}_{\theta,\kappa}} e^{zt_{n}} \left(\delta^{\alpha}_{\tau}(e^{-z\tau}) -A\right)^{-1} \delta^{2}_{\tau}(e^{-z\tau}) \frac{\rho_{2}( e^{-z\tau})}{2} \tau^{3} A - (z^{\alpha}-A)^{-1} z^{-1} A \upsilon dz,\\
I_{2} & = \frac{1}{2\pi i} \int_{\Gamma_{\theta,\kappa}\backslash \Gamma^{\tau}_{\theta,\kappa}} e^{zt_{n}} ( z^{\alpha}-A)^{-1} z^{-1} A\upsilon dz, 
\end{split}
\end{equation*}
\begin{equation*}
\begin{split}
I_{3} & = \frac{1}{2\pi i} \int_{\Gamma^{\tau}_{\theta,\kappa}} e^{zt_{n}} \left(\delta^{\alpha}_{\tau}(e^{-z\tau}) -A\right)^{-1} \delta^{2}_{\tau}(e^{-z\tau}) \frac{\rho_{3}(e^{-z\tau})}{6} \tau^{4} A - (z^{\alpha}-A)^{-1} z^{-2} A  b dz, \quad\\
I_{4} & = \frac{1}{2\pi i} \int_{\Gamma_{\theta,\kappa}\backslash \Gamma^{\tau}_{\theta,\kappa}} e^{zt_{n}} ( z^{\alpha}-A)^{-1} z^{-2}Ab dz.
\end{split}
\end{equation*}
The related noise term is
\begin{equation*}
\begin{split}
I_{5}
& = \frac{\tau}{2\pi i}\int_{\Gamma^{\tau}_{\theta,\kappa}} e^{zt_n} \left( \delta^{\alpha}_{\tau}(e^{-z\tau}) -A\right)^{-1} \delta^{3-\gamma}_{\tau}(e^{-z\tau}) \widetilde{g_{l}} (e^{-z\tau}) dz \\
& \quad - \frac{1}{2\pi i} \int_{\Gamma_{\theta, \kappa}} e^{zt_n} ( z^{\alpha}-A)^{-1} z^{3-\gamma} \widehat{g_{l}}(z) dz,
\end{split}
\end{equation*}
which is estimated in Theorem \ref{SDETheorem3.5}.

From Lemma \ref{SDELemma3.4}, we estimate $I_{1}$ and $I_{3}$ as following
\begin{equation*}
\left\|I_{1} \right\|_{L^2(\Omega)} \leq c\tau^{3} t_{n}^{-3} \left\| \upsilon \right\|_{L^2(\Omega)} \quad {\rm and} \quad \left\|I_{3} \right\|_{L^2(\Omega)} \leq c\tau^{3} t_{n}^{-2} \left\| b \right\|_{L^2(\Omega)}.
\end{equation*}
Using the resolvent estimate~\eqref{SDE1.4}, we estimate $I_{2}$ and $I_{4}$ as following
\begin{equation*}
\left\|I_{2} \right\|_{L^2(\Omega)} \leq c\tau^{3} t_{n}^{-3} \left\| \upsilon \right\|_{L^2(\Omega)} \quad {\rm and} \quad \left\|I_{4} \right\|_{L^2(\Omega)} \leq c\tau^{3} t_{n}^{-2} \left\| b \right\|_{L^2(\Omega)}.
\end{equation*}
Therefore, we have
\begin{equation*}
\mathbb{E} \left\| V_{l}(t_{n})- V_{l}^{n} \right\|^{2}
\leq c\tau^{6} t_{n}^{-6} \left\| \upsilon \right\|_{L^2(\Omega)} + c\tau^{6} t_{n}^{-4} \left\| b \right\|_{L^2(\Omega)} + c \sum_{j=1}^{\infty}\left( \mu_{j}^{l} \right)^{2} \tau^{2\alpha + 2\gamma-1-2\varepsilon}  t_{n}^{2\varepsilon}.
\end{equation*}
The proof is completed.
\end{proof}
\begin{theorem}\label{SDETheorem3.7}
Let $V (t_{n})$ and $V_{l}^{n}$ be the solutions of~\eqref{SDE3.1} and~\eqref{SDE3.4}, respectively.
Let $\upsilon, b \in L^{2}(\Omega)$. 
Then for any small $\varepsilon>0$, it holds
\begin{equation*}
\begin{split}
\mathbb{E} \left\|V(t_{n})-V_{l}^{n}\right\|^{2}
& \leq c\tau^{6} t_{n}^{-6} \left\| \upsilon \right\|_{L^2(\Omega)} + c\tau^{6} t_{n}^{-4} \left\| b \right\|_{L^2(\Omega)}\\
& \quad + c \sum_{j=1}^{\infty} {(\eta_{j}^{l})}^{2} + c \sum_{j=1}^{\infty}\left( \mu_{j}^{l} \right)^{2} \tau^{2\alpha+2\gamma-1-2\varepsilon}  t_{n}^{2\varepsilon}.
\end{split}
\end{equation*}
\end{theorem}
\begin{proof}
By Theorems \ref{SDETheorem2.2}, \ref{SDETheorem3.6} and the triangle inequality, the proof is completed.
\end{proof}

\section{Numerical results}\label{Sec4}
We numerically verify the above theoretical results and the discrete $L^2$-norm ($\|\cdot\|_{l_2}$) is used to measure the numerical errors at the terminal time, e.g., $t=t_N=1$.
We discretize the space direction by the Galerkin finite element method \cite{ThomeeGalerkin2006}.
Here we mainly focus on the time direction convergence order, since the convergence rate of the spatial discretization is well understood \cite{DuNumerical2002}.
The order of the convergence of the numerical results are computed by 
\begin{equation*}
{\rm Convergence ~Rate}=\frac{\ln \left( \left\|u^{N/2}-u^{N} \right\| / \left\|u^{N}-u^{2N} \right\| \right)}{\ln 2}
\end{equation*}
and
\begin{equation*}
\left\| u^{N/2}-u^{N} \right\| := \sqrt{\mathbb{E} \left\| u^{N/2}-u^{N} \right\|_{l_2}^{2}}
\end{equation*}
with $u^N=V^N+ \upsilon$.

Let $T=1$, $\mathcal{O}=(0, 1)$, $\upsilon(x)=\sin \left( x \right) \sqrt{1-x^{2}}$, and $b\left( x \right)= \cos \left( x \right) \sqrt{1-x^{2}} $.
Next, we briefly discuss the implementation of the term involving the noise defined in~\eqref{SDE1.1}.

Let
\begin{equation*}
\sigma_{j}^{l} \left( t \right) = \left\{
\begin{split}
& \sigma_{j} \left( t \right), \quad j \leq l,\\
& 0, \qquad \quad j>l.
\end{split}\right.
\end{equation*}
For simplicity we choose $\sigma _{j} \left( t \right) =j^{-2} $, that is
\begin{equation*}\label{mm4.1}
\frac{d W(t)}{dt} = \sum_{j=1}^{\infty }j^{-2}\dot{\beta}_{j}(t) \varphi_{j}(x),
\end{equation*}
where $\beta_{j}$, $j=1,2,\ldots$ are the independently identically distributed Brownian motions, and $\varphi_{j}$ are the eigenfunctions of the covariance operator $Q$, see \cite{DuNumerical2002}.
In particular, in the one-dimensional case, we have $\varphi_{j}(x)=\sqrt{2}\sin \left( j\pi x \right)$, $j=1,2, \ldots$.

Using BDF$2$ integrals convolution quadrature formula \cite{ChenDiscretized2015,LubichDiscretized1986}, it yields
\begin{equation*}\label{mm4.2}
\begin{split}
\partial_{t}^{-2} \frac{ d W_{l} \left( t_{n} \right)}{d t }  =  \partial_{t}^{-1}  W_{l} \left( t_{n} \right) 
\approx \overline{\tau} \sum_{k=1}^{\overline{n}} w_{\overline{n}-k}^{\left( -1 \right)} \sum_{j=1}^{l} j^{-2} \beta_{j} \left( \overline{t}_{k} \right) \varphi_{j}
\end{split}
\end{equation*}
with $l=100$.
Here Brownian motions $\{\beta_{j} \}_{j=1}^{l}$ can be generated by MATLAB code, see \cite[p.\,395]{SchmidtStochastic2015}.
Since we do not have an explicit representation of the exact value $\partial_{t}^{-1}  W_{l} \left( t_{n},x \right) $, we compare the numerical integrals with a reference obtained on very fine grids with time step size $\overline{\tau} = 2^{-20}$ and $ t_{n} = n \overline{\tau} = n \frac{1}{N} = n \frac{\overline{N}}{N} \overline{\tau} = \overline{n} \, \overline{\tau}=\overline{t}_{\overline{n}} $ with $\overline{n} = n \overline{N}/ N$ and $\overline{N} = 2^{20}$.
All the expected values are computed with 1000 trajectories.
\begin{table}[htb]
\begin{center}
\caption{Convergent order of ID$1$-BDF$2$ method \eqref{SDE2.17}.}\label{table:1}
\begin{tabular}{c c c c c c c} \hline
\multirow{2}{*}{$\gamma$}  &  \multicolumn{3}{c}{$\alpha=1.3$}    &  \multicolumn{3}{c}{$\alpha=1.7$} \\ \cline{2-7}
                           & N=128      & N= 256     & N= 512      & N=128      & N= 256     & N= 512            \\ \hline
\multirow{2}{*}{0.1}       & 3.0075e-03 & 1.6551e-03 & 9.3191e-04  & 1.5179e-03 & 7.7910e-04 & 3.7855e-04  \\
                           &            & 0.8616     & 0.8286      &            & 0.9621     & 1.0413        \\ \hline
\multirow{2}{*}{0.5}       & 9.9941e-04 & 5.1896e-04 & 2.5345e-04  & 7.5818e-04 & 3.8063e-04 & 1.8622e-04  \\
                           &            & 0.9454     & 1.0339      &            & 0.9941     & 1.0313        \\ \hline
\multirow{2}{*}{0.9}       & 4.2796e-04 & 2.1800e-04 & 1.0674e-04  & 4.3302e-04 & 2.1235e-04 & 1.0925e-04  \\
                           &            & 0.9731     & 1.0303      &            & 1.0279     & 0.9588        \\ \hline
\end{tabular}
\end{center}
\end{table}
\begin{table}[htb]
{\begin{center}
\caption{Convergent order of ID$2$-BDF$2$ method \eqref{SDE2.6}.}\label{table:2}
\begin{tabular}{c c c c c c c} \hline
\multirow{2}{*}{$\gamma$}  &  \multicolumn{3}{c}{$\alpha=1.3$}    &  \multicolumn{3}{c}{$\alpha=1.7$} \\ \cline{2-7}
                           & N=128      & N= 256     & N= 512      & N=128      & N= 256     & N= 512        \\ \hline
\multirow{2}{*}{0.1}       & 2.5488e-03 & 1.3806e-03 & 7.5982e-04  & 1.5534e-04 & 4.7273e-05 & 1.7259e-05  \\
                           &            & 0.8845     & 0.8616      &            & 1.7163     & 1.4536        \\ \hline
\multirow{2}{*}{0.5}       & 1.0291e-04 & 3.8766e-05 & 1.5659e-05  & 1.0073e-04 & 2.5715e-05 & 6.4912e-06  \\
                           &            & 1.4085     & 1.3077      &            & 1.9697     & 1.9860        \\ \hline
\multirow{2}{*}{0.9}       & 2.3775e-05 & 6.4497e-06 & 1.7749e-06  & 9.1247e-05 & 2.3142e-05 & 5.8081e-06 \\
                           &            & 1.8821     & 1.8614      &            & 1.9792     & 1.9943        \\ \hline
\end{tabular}
\end{center}}
\end{table}
\begin{table}[ht!]
{\begin{center}
\caption{Convergent order of ID$3$-BDF$3$ method \eqref{SDE3.3}.}\label{table:3}
\begin{tabular}{c c c c c c c} \hline
\multirow{2}{*}{$\gamma$}  &  \multicolumn{3}{c}{$\alpha=1.3$}    &  \multicolumn{3}{c}{$\alpha=1.7$} \\ \cline{2-7}
                           & N=128      & N= 256     & N= 512      & N=128      & N= 256     & N= 512       \\ \hline
\multirow{2}{*}{0.1}       & 1.2903e-02 & 6.8295e-03 & 3.7112e-03  & 8.2492e-04 & 3.3000e-04 & 1.3579e-04  \\
                           &            & 0.9178     & 0.8798      &            & 1.3217     & 1.2811        \\ \hline
\multirow{2}{*}{0.5}       & 8.2490e-04 & 3.3005e-04 & 1.3579e-04  & 4.9274e-05 & 1.4860e-05 & 4.6480e-06  \\
                           &            & 1.3215     & 1.2813      &            & 1.7293     & 1.6767        \\ \hline
\multirow{2}{*}{0.9}       & 4.8742e-05 & 1.4848e-05 & 4.6476e-06  & 6.9838e-06 & 9.8018e-07 & 1.6609e-07  \\
                           &            & 1.7148     & 1.6757      &            & 2.8328     & 2.5610        \\ \hline
\end{tabular}
\end{center}}
\end{table}

 For the linear stochastic  fractional evolution equation with  integrated additive noise, many predominant time-stepping methods lead to  low order error estimates
with $O\left(\tau^{\min\{\alpha + \gamma -1/2,1\}}\right)$,
see Theorem 5.9 in   \cite{KangGalerkin2022}, also see Table \ref{table:1} by   ID$1$-BDF$2$ method. Here the additive noise \eqref{SDE1.5} is regularized by using a one-fold integral-differential (ID$1$) calculus.

To break the first-order barrier in stochastic fractional PDEs,
we establish the convergence properties of the ID$2$-BDF2 method, demonstrating a convergence order of $O(\tau^{\alpha + \gamma -1/2})$ for $1< \alpha + \gamma \leq 5/2$, and $O(\tau^{2})$ for $5/2< \alpha + \gamma <3$,  which is characterized by Theorem \ref{SDETheorem2.10}, see Table \ref{table:2}.

We also establish the convergence order of the  ID$3$-BDF3 method as $O(\tau^{\alpha + \gamma -1/2})$, which is characterized by Theorem \ref{SDETheorem3.7}, see  Table \ref{table:3}.

\end{document}